\documentclass[a4paper]{amsart}

\usepackage[utf8]{inputenc}
\usepackage[english]{babel}
\usepackage{stmaryrd,mathtools,commath,mathrsfs}
\usepackage{bbm} 
\usepackage{amssymb,amsthm}
\numberwithin{equation}{section}
\usepackage{float}
\usepackage{graphicx}
\usepackage{yfonts,verbatim,calrsfs,enumitem}
\usepackage{comment}
\usepackage{todonotes}

\newlist{myenumerate}{enumerate}{1}
\setlist[myenumerate]{label=\textup{(\roman*)},leftmargin=*}

\usepackage{amsmath}
\usepackage{xcolor}
\definecolor{cornellred}{rgb}{0.7,0.11,0.11}

\usepackage{cite}
\usepackage[colorlinks=false, allcolors=cornellred, pdfpagelabels]{hyperref}

\allowdisplaybreaks

\theoremstyle{plain}
\newtheorem{theorem}{Theorem}[section]
\newtheorem{lemma}[theorem]{Lemma}
\newtheorem{proposition}[theorem]{Proposition}
\newtheorem{corollary}[theorem]{Corollary}

\theoremstyle{definition}
\newtheorem{definition}[theorem]{Definition}
\newtheorem*{definition*}{Definition}

\newtheorem{notation}[theorem]{Notation}

\DeclareMathAlphabet{\mathcal}{OMS}{cmsy}{m}{n}

\newcommand{\Cs}{\ensuremath{\mathrm{C}^\ast}}

\newcommand{\id}{\ensuremath{\mathrm{id}}}
\renewcommand{\ker}{\ensuremath{\mathrm{ker}}}

\newcommand{\U}{\ensuremath{\mathcal{U}}}
\newcommand{\K}{\ensuremath{\mathcal{K}}}

\newcommand{\calZ}{\ensuremath{\mathcal{Z}}}
\newcommand{\calU}{\ensuremath{\mathcal{U}}}

\newcommand{\Cus}{\ensuremath{\mathrm{Cu}}}

\newcommand{\ev}{\ensuremath{\mathrm{ev}}}

\newcommand{\cal}[1]{\ensuremath{\mathcal{#1}}}
\newcommand{\leb}{\ensuremath{\mathrm{Leb}}}

\usepackage{tikz}
\usetikzlibrary{arrows,matrix,positioning,automata}
\usepackage{quiver}
\usepackage{pgfplots}

\setlist[enumerate,itemize]{noitemsep} 

\begin{document}

\title[Unital embeddings of the Jiang--Su algebra are not unique]{Unital embeddings of the Jiang--Su algebra are not unique}

\author{Matteo Pagliero}
\address{Department of Mathematics and Computer Science, University of Southern Denmark, Campusvej 55, 5230 Odense, Denmark}
\email{pagliero@imada.sdu.dk}
\thanks{}

\subjclass[2020]{46L05, 46L35, 46L09, 46L35}

\keywords{Jiang--Su algebra, selfless \Cs-algebras, strict comparison, free product}

\date{}

\dedicatory{}

\begin{abstract}
We show that the canonical embeddings of the Jiang--Su algebra $\mathcal Z$ into the universal unital free product $\mathcal Z \ast_{\mathbb C}\mathcal Z$ are not approximately unitarily equivalent.  This answers a question appearing in one of the $99$ problems list of Schafhauser, Tikuisis, and White.
By contrast, we prove that any two such
embeddings into a selfless $\mathrm{C}^\ast$-algebra are strongly asymptotically unitarily equivalent.
\end{abstract}

\maketitle

\tableofcontents


\section{Introduction}
\renewcommand\thetheorem{\Alph{theorem}}

This article is concerned with the uniqueness problem for unital embeddings of the Jiang--Su algebra $\mathcal Z$, which was introduced in \cite{JS99}, into unital \Cs-algebras. More precisely, Schafhauser--Tikuisis--White explicitly
asked in \cite[Problem~LXIII]{STW25} whether any two unital embeddings of $\mathcal Z$
into an arbitrary unital \Cs-algebra are approximately,
or even asymptotically, unitarily equivalent.

In this article, we answer this question negatively using the universal unital free product $\calZ\ast_{\mathbb C}\calZ$ as target, which is in particular a quasidiagonal \Cs-algebra by \cite[Proposition 13]{Boc97}.

\begin{theorem}\label{introthm:nonunique-Z-embeddings}
The two canonical unital embeddings of $\calZ$ into the universal unital free product \Cs-algebra $\calZ\ast_{\mathbb C}\calZ$ are not approximately unitarily equivalent.
\end{theorem}

In fact, in Theorem \ref{thm:non-unique-embeddings} we show that the same conclusion is true if one replaces $\calZ$ with any unital, non-trivial, matricial field (MF) algebra whose ambient matrix algebras have size equal to powers of prime numbers. It is worth noting that the obstruction that allows us to prove Theorem~\ref{introthm:nonunique-Z-embeddings} ultimately arises from algebraic topology, exploiting a connection to \Cs-algebraic regularity similar to that observed in \cite{Vil98,TW09}. With regards to this cross-contamination of mathematical disciplines, the most significant contribution of the article is perhaps that the obstruction can be elegantly packaged in the form of point-wise distance between maps as explained in Lemma~\ref{lem:obstruction-lemma}.
We remark that the example established in Theorem~\ref{introthm:nonunique-Z-embeddings} is in stark contrast with the purely infinite case. Indeed, by \cite[Proposition 4.4]{Gab20} (see also \cite{BG24}), unital embeddings of the Cuntz algebra $\mathcal O_\infty$ \cite{Cun77} into any unital \Cs-algebra are asymptotically unitarily equivalent. 
We also show in Theorem \ref{thm:obstruction-Cuntz} that the same reason underpinning the non-equivalence of $\iota_0$ and $\iota_1$ produces an obstruction to strict comparison of positive elements in $\calZ\ast_{\mathbb C}\calZ$. This seems to be a promising approach to build \Cs-algebras failing strict comparison. 

In the second part of the article, we address the problem of identifying which assumptions on the target yield uniqueness of embeddings of $\calZ$ up to strong asymptotic unitary equivalence.
For a $\mathcal Z$-stable \Cs-algebra $A$, which means that $A\cong A\otimes\calZ$, strong asymptotic unitary
equivalence of unital embeddings $\mathcal Z \to A$ is provided in
\cite{DW09}.
This is also in line with classification theory. In fact, summarising an astonishing theorem that rests on the work of many authors 
\cite{Win14,Win16,GLN20a,GLN20b,EGLN25,TWW17,CETWW21,CGSTW23,Kir95,KP00,Phi00},  simple,
separable, unital, nuclear, $\mathcal Z$-stable \Cs-algebras
that satisfy the UCT are classified by their Elliott invariant.
In absence of $\mathcal Z$-stability, Robert's classification of homomorphisms from
one-dimensional NCCW complexes provides approximate unitary
equivalence of unital embeddings of $\calZ$ for certain regular targets
\cite[Proposition~6.3.1]{Rob12}. More precisely, the codomain is allowed to be simple, unital, non-elementary, of stable rank one \cite{Rie83}, with a unique 2-quasitracial state and strict comparison of positive elements \cite{Bla88}. 
In particular, this applies to \Cs-algebras that are tracial and selfless.

Inspired by the work of Popa \cite{Pop95}, Dykema and R{\o}rdam \cite{DR98,DR00}, in \cite{Rob25} Robert introduced the notion of a 
\emph{selfless} \Cs-probability space that allows one to produce regular \Cs-algebras via the reduced free product construction of Avitzour \cite{Avi82} and Voiculescu \cite{Voi85}. In the same work, he also describes selflessness in the model theoretic framework for \Cs-algebras \cite{FHLRTVW21,Gol23}. A \Cs-probability space $(A,\rho)$, which is a unital \Cs-algebra with a distinguished GNS-faithful state, is
selfless if a nontrivial free product extension of $(A,\rho)$
embeds into a norm ultrapower of $(A,\rho)$ while fixing
the diagonal copy of $A$ (see Section~\ref{section-selfless} for more details). In the tracial case, selflessness
implies simplicity, uniqueness of the trace, stable rank one,
and strict comparison  as observed in \cite[Theorem 3.1]{Rob25}.
Hence, one may regard selflessness as a
means of obtaining regularity properties beyond the nuclear
setting. We remark that an analogous way to obtain regularity was also studied in \cite{BS16}.
Moreover, Robert has very recently
proved a converse to the aforementioned statement: an infinite-dimensional, simple, unital,
monotracial \Cs-algebra with strict comparison by its trace
is selfless \cite[Theorem 1.1]{Rob26}. This completed the general characterisation of selflessness, as for non-tracial \Cs-probability spaces, selflessness is equivalent to being simple and purely infinite by \cite{Ozawa25,Rob25,Gould26}. 
The class of known selfless \Cs-algebras has grown rapidly.
Amrutam--Gao--Kunnawalkam Elayavalli--Patchell used
selflessness to establish strict comparison for certain reduced
group \Cs-algebras, including the striking example $\Cs_r(\mathbb F_2)$ \cite{AGKP25}.
Further group \Cs-algebras examples and methods were developed in
\cite{Vig25,Ozawa25,KPT25,Vig26,GJKPR26,HKEPR25}, with related boundary
and quantitative group-theoretic results in
\cite{Yang25,BF26,Arimoto26}.
Twisted group \Cs-algebras and group extensions are studied
in \cite{RTV25,FFKP26b,Omland26,FG26boundary}.
Moreover, reduced free products, graph products, and amalgamated free
products provide further classes of
examples \cite{HKR25,FFKP26a,GKPT26}.

When the target is selfless and tracial, it follows from the aforementioned result of Robert that there exists a unique unital embedding of $\calZ$ up to approximate unitary equivalence.
Our contribution to this part of \cite[Problem~LXIII]{STW25} is to promote the uniqueness result to \emph{strong asymptotic unitary equivalence} of unital embeddings of $\mathcal Z$.
When the target is purely infinite, we use Gabe's classification of morphisms from \cite{Gab24} to upgrade asymptotic unitary equivalence, which is directly provided by Gabe's work, to \emph{strong} asymptotic unitary equivalence.
Hence, we obtain the following: 

\begin{theorem}\label{introthm:asymptotic-uniqueness}
Let $(A,\rho)$ be a selfless \Cs-probability space.
Then any two unital embeddings $\varphi,\psi\colon\calZ\to A$ are strongly
asymptotically unitarily equivalent, that is, there exists a
norm-continuous path $(u_t)_{t\geq0}\subseteq\mathcal U(A)$,
such that  $u_0=1_A$, and
\[
    \lim_{t\to\infty}
    \|u_t\varphi(z)u_t^*-\psi(z)\|=0
    \qquad\text{for every }z\in\mathcal Z.
\]
\end{theorem}
 

\subsection*{Statement on the use of AI}
In Section~\ref{sec:obstruction}, ChatGPT5.6 was used to prove Lemma \ref{lem:equivariant-zero}, and to give an earlier proof of Lemma \ref{lem:matrix-orbit-gap}. It was also used to produce the maps $\sigma_N$ and $\lambda_N$, and an earlier version of the embeddings $\sigma,\lambda\colon \calZ\to\prod_{\omega}C(\mathcal{PU}(N_n),M_{N_n})$ in Theorem~\ref{thm:non-unique-embeddings}. Subsequently, the author proved, without any use of AI,  Lemma \ref{lem:obstruction-lemma}, and he identified the obstruction to strict comparison in Theorem~\ref{thm:obstruction-Cuntz}, again without any use of AI.

In Section~\ref{sec:asymptotic}, the mathematical content is entirely due to the human author, who used ChatGPT5.6 only to proofread and improve the presentation of his proofs of Lemma~\ref{lem:transfer-unitary-path} and Proposition~\ref{prop:free-product-asymptotic-uniqueness}.

The author takes full responsibility for the content of this article, and confirms that all mathematical contribution from AI was carefully checked by the human author.

\subsection*{Acknowledgements}
The author was supported by the Carlsberg Foundation grant CF24-1144. He would like to express his most sincere gratitude to Jamie Gabe for extremely valuable conversations concerning the topic of this article.


\renewcommand\thetheorem{\arabic{section}.\arabic{theorem}}

\section{Preliminaries}

In this section, we recall some of the terminology, definitions and known facts that are used throughout the article.

\subsection{Selfless \Cs-algebras}\label{section-selfless}
The notions of \Cs-probability space, and selflessness, are central to this article.

\begin{definition}
    A \Cs-probability space $(A,\rho)$ consists of  a unital \Cs-algebra $A$ and a state $\rho$ on $A$ whose GNS-representation is faithful, with the exception of ultrapowers, which are considered as \Cs-probability spaces even though they might be equipped with a non-GNS-faithful state.
\end{definition}

Throughout, by a homomorphism of \Cs-probability spaces, we mean a state preserving unital $\ast$-homomorphism.

\begin{notation}
    The following \Cs-probability spaces will often appear in the article, and hence we fix their notation, and use it without further mention,
    \begin{itemize}[leftmargin=*]
        \item $(C(\mathbb T),\tau_{\leb})$ is the \Cs-probability space generated by one Haar unitary, i.e., $\tau_{\leb}$ is the trace induced by integration against the Lebesgue measure;
        \item $(\calZ,\tau_{\calZ})$ is the Jiang--Su algebra, which was introduced in \cite{JS99}, equipped with its trace, whose uniqueness is proved in \cite[Theorem~2.9]{JS99};
        \item $(\Cs_r(G),\tau_G)$ is the reduced \Cs-algebra of $G$, for a discrete group $G$, and is equipped with the vector state of $\delta_{1_G}\in\ell^2(G)$.
    \end{itemize}
\end{notation}

\begin{definition}
Let $(A_i)_{i\in I}$ be a family of \Cs-algebras, and let
$\omega$ be an ultrafilter on $I$. Consider the \Cs-algebra
\[
    \prod_{i\in I}A_i
    :=
    \big\{
        (a_i)_{i\in I}
        \,\mid \,
        a_i\in A_i,\ 
        \sup_{i\in I}\|a_i\|<\infty
    \big\},
\]
with coordinatewise operations and the supremum norm.
The {(norm) ultraproduct} of $(A_i)_{i\in I}$ with
respect to $\omega$ is the quotient \Cs-algebra 
\[
    \prod_{\omega}A_i
    :=
    \frac{\prod_{i\in I}A_i}{\{
        (a_i)_{i\in I}
        \,\mid \,
        \lim_{i\to\omega}\|a_i\|=0
    \}}.
\]
We denote the class of a bounded family $(a_i)_{i\in I}$
by $[(a_i)_i]_\omega$. Its norm satisfies
\[
    \bigl\|[(a_i)_i]_\omega\bigr\|
    =
    \lim_{i\to\omega}\|a_i\|.
\]
If each $A_i$ is unital, then $\prod_\omega A_i$ is
unital, with unit $[(1_{A_i})_i]_\omega$.

When $A_i=A$ for every $i\in I$, we write $A_\omega$
and call it the {(norm) ultrapower} of $A$.
We identify $A$ with its diagonal copy in $A_\omega$
via the diagonal embedding
\[
    A\longrightarrow A_\omega,
    \qquad
    a\longmapsto[(a)_{i\in I}]_\omega.
\]
    If $\rho$ is a state on $A$, we denote by $\rho_\omega$ the state on $A_\omega$ given by
    \[
    \rho_{\omega}([(a_i)_i]_\omega)=\lim_{i\to\omega}\rho(a_i).
    \]
\end{definition}

\begin{definition}[see {\cite[Definition~2.1 and Theorem~2.6]{Rob25} and \cite[Definition~2.3]{HKEPR25}}]
A \Cs-probability space $(A,\rho)$ is \emph{selfless} if
$A\not\cong\mathbb C$ and there exist a \Cs-probability space $(C,\kappa)$ with $C\not\cong\mathbb C$, an ultrafilter $\omega$, and an embedding
\[
    \sigma\colon (A,\rho)\ast(C,\kappa)
    \longrightarrow (A_\omega,\rho_\omega)
\]
whose restriction to the first factor is the diagonal embedding.
\end{definition}

\subsection{The Cuntz semigroup}\label{subsection-Cuntz-semigroup}

We use the Cuntz semigroup $\Cus(A)$ 
in the sense of \cite{CEI08}, and refer the reader to \cite{APT18,GP26} for a systematic
treatment. The underlying comparison relation on positive elements
originates in \cite{Cun78}. 

\begin{definition}
Let $A$ be a \Cs-algebra, and let $\K$ denote the compact operators on a separable, infinite-dimensional Hilbert space.
For $a,b\in(A\otimes\mathcal K)_+$,
write $a\precsim b$ if there exists a sequence $(x_n)_n\subset A\otimes\mathcal K$
such that $x_n^*bx_n\to a$ in norm, and write $a\sim b$ if both
$a\precsim b$ and $b\precsim a$ occur at the same time. The \emph{Cuntz semigroup} of $A$ is defined as
\[
    \Cus(A):=(A\otimes\mathcal K)_+/{\sim}.
\]
ordered by $[a]\leq[b]$ whenever $a\precsim b$, with addition
$[a]+[b]=[a\oplus b]$. A $\ast$-homomorphism $\varphi\colon A\to B$
induces the map
$\Cus(\varphi)([a])=[(\varphi\otimes\id_{\mathcal K})(a)]$.
\end{definition}

By \cite[Theorem~2]{CEI08}, this defines a functor into the category
of abstract Cuntz semigroups. Its objects are positively ordered abelian monoids
in which increasing sequences have suprema, every element is the
supremum of a sequence $(s_n)_n$ with $s_n\ll s_{n+1}$, and addition
preserves both increasing suprema and the relation $\ll$.
Here $s\ll t$ means that, whenever $t\leq\sup_n t_n$ for an increasing
sequence $(t_n)_n$, one has $s\leq t_n$ for some $n$.
Morphisms preserve zero, addition, order, increasing suprema, and
$\ll$. For unital algebras, we also retain the distinguished element
$[1_A]$, which is preserved by maps induced by unital
$\ast$-homomorphisms.

We record the following result of R{\o}rdam and Winter in the generality that is needed in the present work. The original conclusion is, in fact, stronger, and we refer the reader to \cite{RW10} for the general result.

\begin{theorem}[{see \cite[Proposition 4.2 and Theorem 4.3]{RW10} and \cite{Ell07}}]
    Let $A$ be a \Cs-algebra with stable rank one. Then, $\Cus(A)$ has cancellation of projections, i.e., for any $[a],[b],[p]\in\Cus(A)$, where $p$ is a projection, one has that $[a]+[p] \leq [b]+[p]$ implies $[a]\leq[b]$.
\end{theorem}

\subsection{Strict comparison of positive elements}
Write $\mathrm{QT}_c(A)$ for the cone of lower
semicontinuous, $[0,\infty]$-valued $2$-quasitraces
on $A$, identified with their extensions to
$A\otimes\mathcal K$.
These quasitraces are not required to be densely
defined; see \cite[Section~4.1]{ERS11}.
For $\tau\in\mathrm{QT}_c(A)$ and
$a\in(A\otimes\mathcal K)_+$, put
\[
    d_\tau(a)=\lim_{n\to\infty}\tau(a^{1/n}).
\]
If $A$ is unital, we identify $A$ with the corner
$A\otimes e_{11}$ and write
\[
    \mathrm{QT}_1(A)
    =
    \{\tau\in\mathrm{QT}_c(A):\tau(1_A)=1\}
\]
for its normalised $2$-quasitraces.

\begin{definition}\label{def:strict-comparison}
A \Cs-algebra $A$ has \emph{strict comparison of
positive elements} if, whenever
$a,b\in(A\otimes\mathcal K)_+$ are such that $a$ is contained in the ideal generated by $b$ in $A\otimes\K$ and that $d_\tau(a)<d_\tau(b)$ for every $\tau\in\mathrm{QT}_c(A)$
with $d_\tau(b)=1$, we have $a\precsim b$.

By \cite[Propositions~4.2 and~6.2]{ERS11}, this is
equivalent to almost unperforation of $\mathrm{Cu}(A)$,
that is,
\[
    (n+1)x\leq ny
    \quad\Longrightarrow\quad
    x\leq y
\]
for all $x,y\in\mathrm{Cu}(A)$ and $n\geq1$.
\end{definition}

We recall that Cuntz subequivalence is preserved by
$\ast$-homomorphisms, directly from its definition, and that if $A\subseteq D$ and $a,b\in A_+$ are such that $a\not\precsim b$ as elements of $D$, then $a\not\precsim b$ also as elements of $A$.
For every $\tau\in\mathrm{QT}_c(A)$, the function
$d_\tau$ is order preserving for Cuntz subequivalence
and additive on orthogonal positive elements, i.e., 
\[
    a\precsim b
    \quad\Longrightarrow\quad
    d_\tau(a)\leq d_\tau(b),
\]
and
\[
    ab=0
    \quad\Longrightarrow\quad
    d_\tau(a+b)=d_\tau(a)+d_\tau(b).
\]
These are consequences of
\cite[Proposition~4.2]{ERS11}.

For the latter identity, one may also use $(a+b)^{1/n}=a^{1/n}+b^{1/n}$ and additivity of quasitraces on commuting positive
elements.

If $\tau$ is a faithful tracial state and $a$ is a
nonzero positive contraction in $A_+$, then $0<\tau(a)\leq d_\tau(a)\leq1$.
Indeed, $a\leq a^{1/n}\leq1$ for every $n\geq1$.

Let $\varphi\colon C\to A$ be a unital
$\ast$-homomorphism and let $\tau\in\mathrm{QT}_1(A)$.
Then $\tau\circ\varphi\in\mathrm{QT}_1(C)$, and
functional calculus gives
\[
    d_\tau(\varphi(c))
    =
    d_{\tau\circ\varphi}(c)
    \qquad\text{for all }c\in C_+.
\]
In particular, if
$\varphi\colon\mathcal Z\to A$ is a unital embedding, \cite[Theorem~5.11]{Haa14} implies that $\tau\circ\varphi$ is a tracial state. 
By uniqueness of the trace on $\mathcal Z$, \cite[Theorem~2.9]{JS99}, we obtain
\[
    \tau\circ\varphi=\tau_{\mathcal Z},
    \qquad
    d_\tau(\varphi(c))=d_{\tau_{\mathcal Z}}(c)
\]
for all $c\in\calZ_+$.

A positive element $b$ of a unital \Cs-algebra $A$
is called \emph{full} if it generates $A$ as a
closed two-sided ideal.
For such an element, there exists $m\geq1$ such that $ [1_A]\leq m[b]$ in $\Cus(A)$;
see \cite[Equation~(4.1)]{Ror04}.
If $C$ is simple and
$\varphi\colon C\to A$ is a unital embedding, then
$\varphi(c)$ is full in $A$ for every nonzero
$c\in C_+$.
Indeed, the ideal generated by $\varphi(c)$
intersects $\varphi(C)$ in a nonzero ideal.
It therefore contains $\varphi(C)$, and in
particular $1_A$.

We will use the following observation, which uses
\cite[Propositions~4.2 and~6.2]{ERS11}, and the fullness criterion in \cite[Equation~(4.1)]{Ror04}.

\begin{lemma}\label{lem:strict-comparison-full-elements}
Let $A$ be a unital \Cs-algebra with strict comparison,
and let $a,b\in A_+$, with $b$ full.
If $d_\tau(a)<d_\tau(b)$ for every $\tau\in\mathrm{QT}_1(A)$, then $a\precsim b$.
\end{lemma}

\begin{proof}
Since $b$ is full, $a$ belongs to the ideal generated
by $b$, and there exists $m\geq1$ such that $[1_A]\leq m[b]$.
Fix $\rho\in\mathrm{QT}_c(A)$ with $d_\rho(b)=1$.
Since $b\precsim 1_A$, we have
\[
    1=d_\rho(b)\leq\rho(1_A)
    =d_\rho(1_A)\leq m\,d_\rho(b)=m.
\]
Thus
\[
    \tau=\frac{\rho}{\rho(1_A)}
    \in\mathrm{QT}_1(A).
\]
By assumption,
\[
    d_\rho(a)
    =
    \rho(1_A)d_\tau(a)
    <
    \rho(1_A)d_\tau(b)
    =
    d_\rho(b).
\]
Strict comparison now gives $a\precsim b$.
\end{proof}


\section{The obstruction}\label{sec:obstruction}

For the rest of the present section, for $N>0$, write
\[
    \mathrm \calU(N)=\mathcal U(M_N),
    \qquad
    \cal{PU}(N)
       =\mathrm \calU(N)/(\mathbb T1_N),
\]
where the quotient carries the quotient topology, and denote
the class of $v\in\mathrm \calU(N)$ by $[v]$.
For $k\geq1$, write
\[
    S^{2k-1}=\{\xi\in\mathbb C^k:\|\xi\|=1\}.
\]

The following lemma can obtained as a direct consequence of \cite[Theorem 3.10(i)]{PP13}, which deals with the existence of certain $\mathbb T$-equivariant maps between Stiefel manifolds.

\begin{lemma}
\label{lem:equivariant-zero}
Let $N$ be a power of a prime and let $1\leq r<N$.
Every continuous map
$f\colon\mathcal U(N)\to\mathbb C^r$ satisfying
\[
    f(\lambda v)=\lambda f(v)
    \quad\text{for all }\lambda\in\mathbb T,\
    v\in\mathcal U(N)
\]
vanishes at some point.
\end{lemma}

\begin{proof}
For $1\leq k\leq n$, let $V_k(\mathbb C^n)$ denote the
space of ordered orthonormal $k$-tuples in $\mathbb C^n$,
equipped with the subspace topology and the circle action
\[
    \lambda(\xi_1,\ldots,\xi_k)
       =(\lambda\xi_1,\ldots,\lambda\xi_k).
\]
Identifying a unitary matrix with its ordered columns gives
a $\mathbb T$-equivariant homeomorphism
$\mathcal U(M_N)\cong V_N(\mathbb C^N)$.
Likewise, $V_1(\mathbb C^r)=S^{2r-1}$, where
$S^{2r-1}$ is the unit sphere of $\mathbb C^r$.
Suppose that $f$ never vanishes.
Then
\[
  V_N(\mathbb C^N)\longrightarrow V_1(\mathbb C^r),\quad  v\longmapsto\frac{f(v)}{\|f(v)\|}
\]
defines a continuous $\mathbb T$-equivariant map.
Since $N$ is a prime power and $N>r$, this contradicts
\cite[Theorem~3.10(i)]{PP13}.
\end{proof}


\begin{notation}\label{not:two-embeddings}
    For each $N>0$, define
unital embeddings by
\[
  \sigma_N\colon M_N\to C(\cal{PU}(N),M_N),\quad   \sigma_N(a)([v])=a,  
\]
and 
\[
\lambda_N\colon M_N\to C(\cal{PU}(N),M_N),\quad \lambda_N(a)([v])=v^*av.
\]
Note that the second map does not depend on the representative of
$[v]$.
\end{notation}

\begin{lemma}\label{lem:obstruction-lemma}
Let $N$ be a power of a prime.
For any nonzero elements $a,b\in (M_N)_+$, with $b$ non-invertible, and $x\in  C(\cal{PU}(N),M_N)$, one has
\[
    \|\sigma_N(a)-x\lambda_N(b)x^*\| \
       \geq\|a\|. 
\]
\end{lemma}

\begin{proof}
Let $a,b$ and $x$ be as in the statement.
Choose a unit vector $\xi\in\mathbb C^N$ with
$a\xi=\|a\|\xi$, which exists by the spectral theorem.
Let $r<N$ be the rank of $b$, and identify $\mathrm{ran}(b)$ with $\mathbb C^r$.  
Define the map
\[
   f\colon \calU(N) \to \mathbb C^r,\quad  f(v)=b^{1/2}v\,x([v])^*\xi
\]
and note that it is continuous, and satisfies $f(\lambda v)=\lambda f(v)$ for all $\lambda\in\mathbb T$ and $v\in \calU(N)$. Hence, by Lemma~\ref{lem:equivariant-zero}, there exists $v_0\in\mathrm \calU(N)$ with 
\[
f(v_0)=b^{1/2}v_0x([v_0])^*\xi=0.
\] 
Now, we use this identity to deduce the claimed norm estimate, 
\begin{align*}
   \|\sigma_N(a)-x\lambda_N(b)x^*\|
&=\sup_{[v]\in \cal{PU}(N)}
  \|a-x([v])v^*bvx([v])^*\|\\
&\geq  \|a-x([v_0])v_0^*bv_0x([v_0])^*\|\\
&\geq \|a\xi-x([v_0])v_0^*b^{1/2}(b^{1/2}v_0x([v_0])^*\xi)\|\\
&= \|a\|.
\end{align*}
\end{proof}

We also record here a specific case of the previous lemma.

\begin{lemma}\label{lem:matrix-orbit-gap}
Let $N$ be a power of a prime.
For every nonzero, self-adjoint element $a\in M_N$ and $w\in\mathcal U(C(\cal{PU}(N),M_N))$, we have
\[
    \|w\sigma_N(a)w^*-\lambda_N(a)\|
       =\max\operatorname{sp}(a)-\min\operatorname{sp}(a).
\]
\end{lemma}
\begin{proof}
The assertion is immediate when $a$ is a multiple of the identity matrix.
Since both embeddings are unital, subtracting
$\min\operatorname{sp}(a)1_N$ reduces us to the case where
$a$ is positive, nonzero, and non-invertible. By Lemma \ref{lem:obstruction-lemma}, 
\[
\|w\sigma_N(a)w^*-\lambda_N(a)\|\geq\|a\|.
\]
The reverse inequality follows because both positive
elements lie between $0$ and $\|a\|1_N$. This proves the statement.
\end{proof}

We are ready to prove the following obstruction theorem to approximate unitary equivalence of morphisms for certain MF algebras. Recall that a unital matricial field (MF) algebra is a unital \Cs-algebra $A$ that unitally embeds into an ultraproduct of matrix algebras.

\begin{theorem}\label{thm:non-unique-embeddings}
    Let $A\not\cong\mathbb C$ be a unital \Cs-algebra that embeds unitally into $\prod_{\omega}M_{N_n}$, where $({N_n})_n$ is a sequence of powers of prime numbers.
    Let $\iota_0,\iota_1$ denote the canonical unital embeddings into the universal unital free product $A\ast_{\mathbb C} A$.
    Then,
    \[
    \|w\iota_0(h)w^*-\iota_1(h)\|=1
    \]
    for all   $w\in\calU(A\ast_{\mathbb C}A)$ and all positive, non-invertible, norm one elements $h\in A$.

    In particular, $\iota_0,\iota_1$ are not approximately unitarily equivalent.
\end{theorem}

\begin{proof}
    We identify $A$ with its image in $\prod_{\omega}M_{N_n}$.
    Let $\sigma,\lambda\colon A \to \prod_{\omega}C(\cal{PU}(N_n),M_{N_n})$ be the unital $\ast$-homomorphisms induced by the sequences of maps $(\sigma_{N_n})_n$ and $(\lambda_{N_n})_n$ defined in Notation \ref{not:two-embeddings}.
    Note that, for any positive norm one, non-invertible $h\in A$,
\[
    \max\operatorname{sp}(h)-\min\operatorname{sp}(h)
       =\|h\|+\|1_A-h\|-1=1,
\]
where the first identity holds for any positive contraction in any unital \Cs-algebra.
Consequently, lifting $h$ to a sequence of positive, norm one elements $(h_n)_n$ in $\prod_{n=1}^\infty M_{N_n}$, we have that 
\[
    \lim_{n\to\omega}
    \left(
       \max\operatorname{sp}(h_n)
       -\min\operatorname{sp}(h_n)
    \right)=1.
\]
Consider now any unitary $u\in \prod_{\omega}C(\cal{PU}(N_n),M_{N_n})$, and choose a representative sequence
$(u_n)_n$ of $u$ with $u_n\in\mathcal U(C(\cal{PU}(N_n),M_{N_n}))$.
It follows from Lemma~\ref{lem:matrix-orbit-gap} that 
\begin{align*}
    \|u\sigma(h)u^*-\lambda(h)\|
      &=
      \lim_{n\to\omega}
      \|u_n\sigma_{N_n}(h_n)u_n^*
                 -\lambda_{N_n}(h_n)\| \\
    &=\lim_{n\to\omega}\left(
       \max\operatorname{sp}(h_n)
       -\min\operatorname{sp}(h_n)
    \right)\\
    &=1.
\end{align*}
Use the universal property of $A\ast_{\mathbb C}A$ to obtain a unital $\ast$-homomorphism
\[
q\colon A\ast_{\mathbb C}A\to \prod_{\omega}C(\cal{PU}(N_n),M_{N_n})
\]
with $q\circ \iota_0=\sigma$ and
$q\circ \iota_1=\lambda$.
For every $w\in\mathcal U(A\ast_{\mathbb C}A)$,
\[
    \|w\iota_0(h)w^*-\iota_1(h)\|
       \geq
    \|q(w)\sigma(h)q(w)^*-\lambda(h)\|=1.
\]
\end{proof}

Observe that UHF algebras of type $M_{p^\infty}$ for some prime $p$, and thus any unital \Cs-algebra that embeds unitally into one such \Cs-algebra, satisfy the assumptions of Theorem \ref{thm:non-unique-embeddings}.
Now,  $\calZ$ embeds unitally into all such UHF \Cs-algebras because they are $\calZ$-stable by \cite[Corollary 6.3]{JS99}. Therefore, Theorem \ref{introthm:nonunique-Z-embeddings} follows immediately from Theorem \ref{thm:non-unique-embeddings}.

We now turn to the obstruction to strict comparison, which illustrates the true strength of the obstruction provided by Lemma \ref{lem:obstruction-lemma}.

\begin{theorem}\label{thm:obstruction-Cuntz}
Let $(N_n)_n$ be a sequence of powers of primes such that $\calZ$ embeds unitally into $\prod_{\omega}M_{N_n}$ for some free ultrafilter $\omega$. 
Let $\sigma,\lambda\colon \calZ \to \prod_{\omega}C(\cal{PU}(N_n),M_{N_n})$ be the unital $\ast$-homomorphisms induced by the sequences of maps $(\sigma_{N_n})_n$ and $(\lambda_{N_n})_n$ defined in Notation \ref{not:two-embeddings}.
Then every \Cs-subalgebra $B$ in the inclusion
\[
    C^*(\sigma(\mathcal Z),\lambda(\mathcal Z))
    \subseteq B\subseteq \prod_{\omega}C(\cal{PU}(N_n),M_{N_n})
\]
fails strict comparison of positive elements.
\end{theorem}

\begin{proof}
Choose a positive contraction $h\in \calZ$ with
$\operatorname{sp}(h)=[0,1]$, and pairwise orthogonal positive contractions
$f_1,f_2,f_3\in C([0,1])$ of norm one.
Then, use continuous functional calculus to define 
\[
    a_j=f_j(h)\in \calZ,
    \qquad j=1,2,3.
\]
These are pairwise orthogonal positive contractions
of norm one.
Now, since $\tau_{\mathcal Z}$ is faithful,
\[
    d_{\tau_{\mathcal Z}}(a_j)>0
    \qquad\text{for }j=1,2,3.
\]
After relabelling, we may assume that
\[
    d_{\tau_{\mathcal Z}}(a_1)
    \leq d_{\tau_{\mathcal Z}}(a_j)
    \qquad\text{for }j=2,3.
\]
Set $ a=a_1$, and $ b=a_2+a_3$.
Then $a,b\in \calZ$ are positive contractions of norm one,
$ab=0$, and orthogonality gives
\begin{equation*}
\begin{aligned}
    d_{\tau_{\mathcal Z}}(a)
    <
    d_{\tau_{\mathcal Z}}(a_2)
    +d_{\tau_{\mathcal Z}}(a_3)
    =d_{\tau_{\mathcal Z}}(b).
\end{aligned}
\end{equation*}

Let $\tau$ be any normalized $2$-quasitrace on $B$.
Since $\mathcal Z$ is exact,
\cite[Theorem~5.11]{Haa14} implies that
$\tau\circ\sigma$ and $\tau\circ\lambda$ are tracial
states on $\mathcal Z$.
By uniqueness of its trace,
\[
    \tau\circ\sigma
    =\tau_{\mathcal Z}
    =\tau\circ\lambda.
\]
It follows that 
that
\begin{equation*}
    d_\tau(\sigma(a))
    =
    d_{\tau_{\mathcal Z}}(a)
    <
    d_{\tau_{\mathcal Z}}(b)
    =
    d_\tau(\lambda(b)).
\end{equation*}

We show that nevertheless
$\sigma(a)$ is not Cuntz subequivalent to $\lambda(b)$, even inside the larger \Cs-algebra $\prod_{\omega}C(\cal{PU}(N_n),M_{N_n})$.
Lift $a$ and $b$ to sequences of positive contractions $(c_n)_n$ and $(d_n)_n$, respectively, which we may choose so that they satisfy $c_nd_n=0$. The last condition is justified by taking a lift $(h_n)_n$ of $h$, and $c_n=f_1(h_n)$, $d_n=(f_2+f_3)(h_n)$.
Since $\sigma$ is isometric,
\[
    \lim_{n\to\omega}\|c_n\|=\lim_{n\to\omega}\|\sigma_{N_n}(c_n)\|=\|\sigma(a)\|=1.
\]
Note that, whenever $c_n\neq0$, the relation $c_nd_n=0$ implies
that $d_n$ is a non-invertible matrix of size $N_n$. Moreover, if either $c_n=0$ or $d_n=0$, the estimate we want to show below is immediate. 

Fix $x=[(x_n)_n]_\omega\in \prod_{\omega}C(\cal{PU}(N_n),M_{N_n})$.
Applying Lemma \ref{lem:obstruction-lemma}, we obtain that
\[
\begin{aligned}
    \|\sigma(a)-x\lambda(b)x^*\|
    &=
    \lim_{n\to\omega}
    \|\sigma_{N_n}(c_n)
        -x_n\lambda_{N_n}(d_n)x_n^*\|\\
    &\geq
    \lim_{n\to\omega}\|c_n\|\\
    &=1.
\end{aligned}
\]
Hence, let $y$ be any element in the stabilisation of $\prod_{\omega}C(\cal{PU}(N_n),M_{N_n})$, and $\widetilde{x}\in \prod_{\omega}C(\cal{PU}(N_n),M_{N_n})$ the element satisfying $\widetilde x\otimes e_{11}=(1\otimes e_{11})y(1\otimes e_{11})$.
Recall that we identify $\prod_{\omega}C(\cal{PU}(N_n),M_{N_n})$ with the corner $\big(\prod_{\omega}C(\cal{PU}(N_n),M_{N_n})\big)\otimes e_{11}$ inside $\big(\prod_{\omega}C(\cal{PU}(N_n),M_{N_n})\big)\otimes\mathcal K$.
We then have that 
\begin{align*}
\|\sigma(a)-y\lambda(b)y^*\|&\geq\|\sigma(a)-(1\otimes e_{11})y\lambda(b)y^*(1\otimes e_{11})\| \\
&=\|\sigma(a)-\widetilde x\lambda(b){\widetilde x}^*\|\\
&=1
\end{align*}
and thus $\sigma(a)$ is not Cuntz subequivalent to $\lambda(b)$ in $\prod_{\omega}C(\cal{PU}(N_n),M_{N_n})$, and therefore also inside $B$.
Since \(b\neq0\), simplicity of \(\mathcal Z\) and unitality of \(\lambda\) imply that \(\lambda(b)\) is full in \(B\). If \(B\) had strict comparison, Lemma~\ref{lem:strict-comparison-full-elements} would give \(\sigma(a)\precsim\lambda(b)\), contradicting the preceding conclusion.
\end{proof}

\begin{corollary}\label{cor:non-strict-comparison-ZZ}
The universal unital free product
$\mathcal Z*_{\mathbb C}\mathcal Z$ does not have
strict comparison of positive elements.
\end{corollary}

\begin{proof}
Denote the canonical unital embeddings of $\calZ$ into  $\mathcal Z\ast_{\mathbb C}\mathcal Z$ by $\iota_0,\iota_1$. Here, we view $\calZ$ unitally embedded into $\prod_{\omega}M_{N_n}$ for a sequence of powers of primes $(N_n)_n$, and let $\sigma,\lambda\colon \calZ \to \prod_{\omega}C(\cal{PU}(N_n),M_{N_n})$ denote the unital $\ast$-homomorphisms induced by the sequences of maps $(\sigma_{N_n})_n$ and $(\lambda_{N_n})_n$ defined in Notation \ref{not:two-embeddings}.
By (the proof of) Theorem \ref{thm:obstruction-Cuntz}, we can find positive elements satisfying 
\[
    d_{\tau_{\mathcal Z}}(a)
    <d_{\tau_{\mathcal Z}}(b),
    \quad \text{and}\quad 
    \sigma(a)\not\precsim\lambda(b).
\]
The universal property of  $\mathcal Z\ast_{\mathbb C}\mathcal Z$ yields a unital $\ast$-homomorphism
\[
q\colon \mathcal Z\ast_{\mathbb C}\mathcal Z\to \prod_{\omega}C(\cal{PU}(N_n),M_{N_n})
\]
satisfying $q\circ\iota_0=\sigma$ and $q\circ\iota_1=\lambda$.
Since $\ast$-homomorphisms preserve Cuntz subequivalence,
\[
    \iota_0(a)\not\precsim\iota_1(b).
\]
On the other hand, every normalized $2$-quasitrace
$\tau$ on $\mathcal Z\ast_{\mathbb C}\mathcal Z$ restricts to $\tau_{\mathcal Z}$ on
both copies, and hence
\[
    d_\tau(\iota_0(a))
    =d_{\tau_{\mathcal Z}}(a)
    <d_{\tau_{\mathcal Z}}(b)
    =d_\tau(\iota_1(b)).
\]
Since $b\neq0$ and $\mathcal Z$ is simple,
$\iota_1(b)$ is full in $\mathcal Z\ast_{\mathbb C}\mathcal Z$, 
Lemma~\ref{lem:strict-comparison-full-elements}
therefore implies that $\mathcal Z\ast_{\mathbb C}\mathcal Z$ cannot have strict comparison.
\end{proof}


\section{Embeddings of $\calZ$ into selfless \Cs-algebras}\label{sec:asymptotic}

In this section, we show that unital embeddings of $\calZ$ into selfless \Cs-algebras are unique up to strong asymptotic unitary equivalence. We note that in the tracial setting, this upgrades the previous uniqueness up to approximate unitary equivalence from \cite{Rob12}. In the non-tracial case, we first find an asymptotic unitary equivalence using Gabe's work from \cite{Gab24}, and then upgrade it to a \emph{strong} asymptotic unitary equivalence.

\begin{definition}
Let $A$ and $B$ be unital \Cs-algebras.
Two unital $\ast$-homomorphisms
$\varphi,\psi\colon A\to B$ are said to be \emph{asymptotically
unitarily equivalent} if there exists a norm-continuous path
$(u_t)_{t\geq0}\subseteq\mathcal U(B)$ such that 
\[
    \lim_{t\to\infty}
    \|u_t\varphi(a)u_t^*-\psi(a)\|=0
\]
for all $a\in A$.

If, additionally, the unitary path satisfies $u_0=1_B$, then $\varphi$ and $\psi$ are said to be \emph{strongly asymptotically unitarily equivalent}.
\end{definition}

We first prove the following uniqueness result for unital embeddings of $\calZ$ into separable, unital, stable rank one \Cs-algebras with vanishing $K_1$-obstruction. It will be used as a reduction theorem to pass from approximate unitary equivalence to strong asymptotic unitary equivalence.

\begin{theorem}\label{thm:Cu-asymptotic-uniqueness}
Let $C$ be a separable, unital \Cs-algebra of stable rank one,
and suppose that $K_1(J)=0$ for every closed two-sided ideal
$J\subseteq C$.
Then any two unital
$\ast$-homomorphisms $\varphi,\psi\colon\mathcal Z\to C$ that satisfy $\Cus(\varphi)=\Cus(\psi)$ are
strongly asymptotically unitarily equivalent.
\end{theorem}

\begin{proof}
Put $j:=\mathrm{Cu}(\varphi)=\mathrm{Cu}(\psi)$ and define
\[
    E:=\{f\in C([0,1],C)\mid
    f(0)\in\varphi(\mathcal Z),\ 
    f(1)\in\psi(\mathcal Z)\}.
\]
Observe that $E$ is isomorphic to the pullback \Cs-algebra $C([0,1],C) \oplus_{C\oplus C}(\calZ\oplus\calZ)$ along the maps $\ev_0\oplus\ev_1$ and $\varphi\oplus\psi$.
Each fibre has stable rank one and vanishing $K_1$ for every
closed two-sided ideal. Note also that $E$ is a continuous field over $[0,1]$ with fibres isomorphic to $C$ for $0<t<1$, and to $\calZ$ at the endpoints. Hence, both $E$ and $C([0,1],C)$ have
stable rank one by \cite[Theorems 1.9 and 1.11]{ABPP14}.

By \cite[Corollaries 2.7 and 3.5]{APS11}, there is a natural identification
\[
    \mathrm{Cu}(E)\cong
    \left\{
    (x_0,x_1,f)\in
    \mathrm{Cu}(\mathcal Z)^2
    \times\operatorname{Lsc}([0,1],\mathrm{Cu}(C))\mid
    f(0)=j(x_0),\ f(1)=j(x_1)
    \right\}.
\]
Under this identification, define
\[
    \Lambda\colon\mathrm{Cu}(\mathcal Z)\to\mathrm{Cu}(E),
    \qquad
    \Lambda(x)=(x,x,[s\mapsto j(x)]).
\]
This is a morphism in the category of abstract Cuntz semigroups. Indeed, zero, addition,
order, and increasing suprema are computed coordinatewise.
The constant-function map
$\mathrm{Cu}(C)\to\operatorname{Lsc}([0,1],\mathrm{Cu}(C))$
is induced by the inclusion $C\to C([0,1],C)$, and therefore
preserves compact containment. Thus, if $x'\ll x$, each
coordinate of $\Lambda(x')$ is compactly contained in the
corresponding coordinate of $\Lambda(x)$. Given an increasing
sequence of triples whose supremum dominates $\Lambda(x)$,
one may choose a common index at which all three coordinates
dominate those of $\Lambda(x')$. This proves that
$\Lambda(x')\ll\Lambda(x)$.
Moreover, $\Lambda([1_{\mathcal Z}])=[1_E]$.

By \cite[Theorems 1.0.1 and 3.2.2(i)]{Rob12}, there exists a
$\ast$-homomorphism $\Theta\colon\mathcal Z\to E$ inducing
$\Lambda$. Put $p:=\Theta(1_{\mathcal Z})$. Since
$[p]=[1_E]$ and
\[
    [p]+[1_E-p]=[1_E]=[p],
\]
cancellation of projections gives $[1_E-p]=0$, and hence
$p=1_E$. Thus $\Theta$ is unital.
Set $\theta_s:=\operatorname{ev}_s\circ\Theta$, and observe that
\[
    \mathrm{Cu}(\theta_s)=j,\quad
    \theta_0=\varphi\circ\alpha_0,\quad
    \theta_1=\psi\circ\alpha_1,
\]
for all $s\in[0,1]$, where $\alpha_0,\alpha_1\colon\mathcal Z\to\mathcal Z$ are some unital endomorphisms.

If we view $\Theta$ as a homomorphism into $C([0,1],C)$, and using \cite[Corollary 2.7]{APS11}, we see that its Cuntz-semigroup map agrees
with that of the constant homomorphism
$z\mapsto[s\mapsto\theta_0(z)]$.
Applying \cite[Theorems 1.0.1 and 3.2.2(i)]{Rob12} in
$C([0,1],C)$, we obtain norm-continuous unitary paths
$(u_{n,s})_{s\in[0,1]}\subseteq\mathcal U(C)$, for $n\geq0$, such that
\[
    \sup_{s\in[0,1]}
    \|u_{n,s}\theta_0(z)u_{n,s}^*-\theta_s(z)\|
    \xrightarrow{n\to\infty}0
\]
for all $z\in\mathcal Z$.
Put $v_{n,s}:=u_{n,s}u_{n,0}^*$. Then $v_{n,0}=1_C$ and
\begin{align*}
    e_n(z)&:=\sup_{s\in[0,1]}
    \|v_{n,s}\theta_0(z)v_{n,s}^*-\theta_s(z)\| \\
    &\leq \sup_{s\in[0,1]} (
    \|u_{n,s}(u_{n,0}^*\theta_0(z)u_{n,0}-\theta_0(z))u_{n,s}^*\| + \|u_{n,s}\theta_0(z)u_{n,s}^*-\theta_s(z)\|)\\
    &= 
    \|u_{n,0}^*\theta_0(z)u_{n,0}-\theta_0(z)\| + \sup_{s\in[0,1]}\|u_{n,s}\theta_0(z)u_{n,s}^*-\theta_s(z)\|
\end{align*}
and thus $e_n(z) \to0$ as $n\to\infty$ for all $z\in\calZ$.
Define $c_{n,s}:=v_{n,s}^*v_{n+1,s}$, and note that  
\begin{align*}
    \|[c_{n,s},\theta_0(z)]\|
    &=\|v_{n,s}
        \bigl(c_{n,s}\theta_0(z)-\theta_0(z)c_{n,s}\bigr)
        v_{n+1,s}^*\|\\
    &=\bigl\|
        v_{n,s}v_{n,s}^*v_{n+1,s}\theta_0(z)v_{n+1,s}^*
        -v_{n,s}\theta_0(z)v_{n,s}^*
        v_{n+1,s}v_{n+1,s}^*
      \bigr\|\\
    &=\|v_{n+1,s}\theta_0(z)v_{n+1,s}^*
        -v_{n,s}\theta_0(z)v_{n,s}^*\|\\
    &=\bigl\|
        \bigl(v_{n+1,s}\theta_0(z)v_{n+1,s}^*-\theta_s(z)\bigr)
        +\bigl(\theta_s(z)-v_{n,s}\theta_0(z)v_{n,s}^*\bigr)
      \bigr\|\\
    &\leq
        \|v_{n+1,s}\theta_0(z)v_{n+1,s}^*-\theta_s(z)\|
        +\|v_{n,s}\theta_0(z)v_{n,s}^*-\theta_s(z)\|\\
    &\leq e_{n+1}(z)+e_n(z)
\end{align*}
for all $s\in[0,1]$, $n\geq0$, and $z\in\calZ$.
Since $v_{n,0}=1_C$ for every $n$, we have that
\[
    c_{n,0}=1_C,\qquad
    c_{n,1}=v_{n,1}^*v_{n+1,1}.
\]
Thus, setting, for each $n\geq1$, $s\in[0,1]$,
\[
    r_{n+s}:=v_{n,1}c_{n,s}
\]
we have defined a continuous unitary path on $[n,n+1]$ with endpoints
\begin{align*}
    r_n
    &=v_{n,1}c_{n,0}=v_{n,1},\\
    r_{n+1}
    &=v_{n,1}c_{n,1}
      =v_{n,1}v_{n,1}^*v_{n+1,1}
      =v_{n+1,1}.
\end{align*}
The consecutive paths therefore agree at their common endpoints.
Together with $r_t:=v_{1,t}$ for $t\in[0,1]$, they define a
norm-continuous unitary path $(r_t)_{t\geq0}$ with $r_0=1_C$.
Now, for every $n\geq1$, $s\in[0,1]$, and $z\in\mathcal Z$,
\begin{align*}
    \|r_{n+s}\theta_0(z)r_{n+s}^*-\theta_1(z)\| &=\|v_{n,1}c_{n,s}\theta_0(z)c_{n,s}^*v_{n,1}^*
        -\theta_1(z)\|\\
    &\leq
    \|v_{n,1}
        \bigl(c_{n,s}\theta_0(z)c_{n,s}^*-\theta_0(z)\bigr)
        v_{n,1}^*\|\\
    & \quad +
    \|v_{n,1}\theta_0(z)v_{n,1}^*-\theta_1(z)\|\\
    &=
    \|[c_{n,s},\theta_0(z)]\|
    +\|v_{n,1}\theta_0(z)v_{n,1}^*-\theta_1(z)\|\\
    &\leq
    \bigl(e_n(z)+e_{n+1}(z)\bigr)+e_n(z)\\
    &=2e_n(z)+e_{n+1}(z).
\end{align*}
Since the bound is independent of $s$, we obtain
\[
    \sup_{s\in[0,1]}
    \|r_{n+s}\theta_0(z)r_{n+s}^*-\theta_1(z)\|
    \leq2e_n(z)+e_{n+1}(z)
    \xrightarrow{n\to\infty}0.
\]
Thus $\theta_0$ and $\theta_1$ are strongly asymptotically
unitarily equivalent.

Finally, since $\mathcal Z$ is strongly self-absorbing and
$K_1$-injective (in fact, all strongly self-absorbing \Cs-algebras have this property by \cite{Win11}), \cite[Theorem~2.2]{DW09} gives norm-continuous
paths $(w_{i,t})_{t\geq0}\subseteq\mathcal U(\mathcal Z)$,
for $i=0,1$, such that $w_{i,0}=1_{\mathcal Z}$ and
\[
    \lim_{t\to\infty}
    \|w_{i,t}\alpha_i(z)w_{i,t}^*-z\|=0
\]
for all $z\in\calZ$.
Define, for $t\ge0$,
\[
    a_t:=\varphi(w_{0,t}^*),
    \qquad
    b_t:=\psi(w_{1,t}).
\]
These are norm-continuous unitary paths in $C$ with
$a_0=b_0=1_C$. Using
$\theta_0=\varphi\circ\alpha_0$ and
$\theta_1=\psi\circ\alpha_1$, we obtain
\begin{align*}
    \|a_t\varphi(z)a_t^*-\theta_0(z)\|
    &=\|\varphi(w_{0,t}^*zw_{0,t}-\alpha_0(z))\|\\
    &=\|w_{0,t}^*zw_{0,t}-\alpha_0(z)\|\\
    &=\|z-w_{0,t}\alpha_0(z)w_{0,t}^*\|
    \xrightarrow{t\to\infty}0,
\end{align*}
and
\begin{align*}
    \|b_t\theta_1(z)b_t^*-\psi(z)\|
    &=\|\psi(w_{1,t}\alpha_1(z)w_{1,t}^*-z)\|\\
    &=\|w_{1,t}\alpha_1(z)w_{1,t}^*-z\|
    \xrightarrow{t\to\infty}0
\end{align*}
for every $z\in\mathcal Z$.
The path $(b_tr_ta_t)_{t\geq0}$ implements the required
strong asymptotic unitary equivalence between $\varphi$ and
$\psi$.
\end{proof}

\begin{lemma}\label{lem:free-Z-path}
The canonical free factor embeddings
\[
    \iota_1,\iota_2\colon\mathcal Z\longrightarrow
    (\mathcal Z,\tau_{\mathcal Z})
    \ast(\mathcal Z,\tau_{\mathcal Z})
\]
are strongly asymptotically unitarily equivalent.
\end{lemma}

\begin{proof}
Put
\[
    (D,\tau_D):=
    (\mathcal Z,\tau_{\mathcal Z})
    \ast(\mathcal Z,\tau_{\mathcal Z}).
\]
By \cite[Theorems 3.1, 4.2, and 5.2]{Rob25}, $D$ is simple,
has stable rank one and strict comparison, and has a unique
$2$-quasitracial state. It is moreover separable and non-elementary.
The reduced free-product exact sequence in $K$-theory
that follows from \cite[Theorem 3.1 and Corollary 4.12]{FG} (see also \cite{Ger97,Tho03}) gives
\[
    K_1(D)\cong
    \ker\bigl(\mathbb Z\to\mathbb Z\oplus\mathbb Z,\ 
    n\mapsto(n,-n)\bigr)=0.
\]
Since $D$ is simple, and $\iota_1$, $\iota_2$ are
approximately unitarily equivalent by
\cite[Proposition 6.3.1]{Rob12}, they induce the same map of Cuntz semigroups, and the conclusion follows from
Theorem~\ref{thm:Cu-asymptotic-uniqueness}.
\end{proof}


\begin{lemma}\label{lem:product-Z}
Let $(A,\rho)$ be a selfless \Cs-probability space.
Then there exist an ultrafilter $\omega$ on a set $I$ and
a state-preserving unital embedding
\[
    \Theta\colon (A,\rho)*(\calZ,\tau_{\calZ})
    \longrightarrow (A_\omega,\rho_\omega)
\]
whose restriction to the first free factor is the diagonal
embedding of $A$.
\end{lemma}

\begin{proof}
By \cite[Theorem~2.6(v)]{Rob25}, there exist an ultrafilter
$\omega$ and a state-preserving unital embedding
\[
    \sigma\colon
    (A,\rho)\ast(\Cs_r(\mathbb F_\infty),\tau_{\mathbb F_\infty})
    \longrightarrow (A_\omega,\rho_\omega)
\]
whose restriction to the first free factor is the diagonal embedding of $A$.
By \cite[Proposition~6.3.1]{Rob12}, and the paragraph following it,
there exists a unital embedding
\[
    j\colon\calZ\longrightarrow\Cs_r(\mathbb F_\infty).
\]
Since $\calZ$ has a unique tracial state, $j$ is trace preserving.
By \cite[Theorem~1.3]{BD01} $\id_A$ and $j$ induce a unital state preserving homomorphism 
\[
\id_A\ast j \colon (A,\rho)\ast (\calZ,\tau_{\calZ}) \longrightarrow (A,\rho)\ast(\Cs_r(\mathbb F_\infty),\tau_{\mathbb F_\infty}).
\]
Therefore, $\Theta:=\sigma\circ(\operatorname{id}_A*j)$ is the required unital state preserving embedding.
\end{proof}

\begin{lemma}\label{lem:transfer-unitary-path}
Let $A\subseteq E$ be a unital inclusion of \Cs-algebras.
Suppose that, for some ultrafilter $\omega$ on a set $I$,
there exists a unital embedding
\[
    \sigma\colon E\longrightarrow A_\omega
\]
whose restriction to $A$ is the diagonal embedding.
Let $\varphi,\psi\colon\calZ\to A$ be unital embeddings,
let $F,H$ be nonempty finite subsets of the unit ball
of $\calZ$, and let $\varepsilon,\eta>0$.

Suppose that $(v_t)_{t\in[0,1]}\subseteq\mathcal U(E)$
is a norm-continuous path with $v_0=1_E$ and
\[
    \sup_{t\in[0,1]}\max_{z\in F}
    \|v_t\varphi(z)v_t^*-\psi(z)\|<\varepsilon,
    \qquad
    \max_{z\in H}\|v_1\varphi(z)v_1^*-\psi(z)\|<\eta.
\]
Then there exists a norm-continuous path
$(w_t)_{t\in[0,1]}\subseteq\mathcal U(A)$ with $w_0=1_A$
such that 
\[
    \sup_{t\in[0,1]}\max_{z\in F}
    \|w_t\varphi(z)w_t^*-\psi(z)\|<\varepsilon,
    \qquad
    \max_{z\in H}\|w_1\varphi(z)w_1^*-\psi(z)\|<\eta.
\]
\end{lemma}

\begin{proof}
Put
\[
    M:=\sup_{t\in[0,1]}\max_{z\in F}
    \|v_t\varphi(z)v_t^*-\psi(z)\|,
\]
and choose $0<\gamma<1$ such that $M+3\gamma<\varepsilon$.
Since $(v_t)$ is a continuous path, we may choose a partition $0=t_0<\cdots<t_m=1$ such that
\[
    \|v_{t_{j+1}}-v_{t_j}\|<\frac{\gamma}{2}
\]
for all $0\leq j<m$.
Now, we can lift each unitary $\sigma(v_{t_j})\in A_\omega$ to a family of unitaries
$(r_{j,i})_{i\in I}$ in $A$; in particular, we may choose $r_{0,i}=1_A$ for all $i\in I$.

Since $\sigma$ is isometric and its restriction to $A$ is the diagonal embedding, for every $m\in\mathbb N$, we may find an index $i\in I$ such that, writing $r_j:=r_{j,i}$,
\[
    \max_{0\leq j<m}\|r_{j+1}-r_j\|<\gamma, \quad
    \max_{0\leq j\leq m}\max_{z\in F}\|r_j\varphi(z)r_j^*-\psi(z)\|
    <M+\gamma,
\]
and
\[
    \max_{z\in H}\|r_m\varphi(z)r_m^*-\psi(z)\|<\eta.
\]
Next, for each $0\leq j<m$, let $h_j=\log(r_j^*r_{j+1})$
using the principal logarithm, which is well-defined because $-1\not\in\mathrm{sp}(r_j^*r_{j+1})$.
For each $0\leq j< m$, define the continuous path $w^{(j)}\colon[t_j,t_{j+1}] \to \mathcal U(A)$ by 
\[
   [t_j,t_{j+1}]\ni s \mapsto w_s^{(j)}= r_j\exp\bigg(\frac{s-t_j}{t_{j+1}-t_j}h_j\bigg).
\]
Note that each $w^{(j)}$ starts at $r_j$ and ends at $r_{j+1}$, and thus we may define $w\colon [0,1] \to \U(A)$ to be the continuous path that is piecewise given by $w^{(j)}$ on $[t_j,t_{j+1}]$.
Therefore, $w$ starts at $1_A$ and ends at $r_m$.
Functional calculus then gives
\[
    \bigg\|r_j\exp\bigg(\frac{s-t_j}{t_{j+1}-t_j}h_j\bigg)-r_j\bigg\|
    \leq\|r_{j+1}-r_j\|<\gamma.
\]
Moreover, as $F$ lies in the unit ball of $\mathcal Z$, we have that
\[
    \sup_{t\in[0,1]}\max_{z\in F}
    \|w_t\varphi(z)w_t^*-\psi(z)\|
    \leq M+3\gamma<\varepsilon.
\]
The estimate  
\[
    \max_{z\in H}\|w_1\varphi(z)w_1^*-\psi(z)\|<\eta
\]
follows from the fact that $w_1=r_m$.
\end{proof}

\begin{proposition}\label{prop:free-product-asymptotic-uniqueness}
Let $A$ be a unital \Cs-algebra with a faithful tracial
state $\tau$. Suppose that, for some ultrafilter $\omega$
on a set $I$, there exists a unital embedding
\[
    \sigma\colon (A,\tau)*(\calZ,\tau_{\calZ})
    \longrightarrow A_\omega
\]
whose restriction to the first free factor is the diagonal
embedding of $A$.
Then any two unital embeddings $\calZ\to A$ are strongly
asymptotically unitarily equivalent.
\end{proposition}

\begin{proof}
Let $\varphi,\psi\colon\calZ\to A$ be unital embeddings, and write 
\[
    \iota_1,\iota_2\colon\calZ\to(\calZ,\tau_{\calZ})*(\calZ,\tau_{\calZ})
\]
for the canonical first and second free-factor embeddings. We also denote by $\kappa\colon\calZ\to (A,\tau)*(\calZ,\tau_{\calZ})$ the second
free-factor embedding in this reduced free product, while we omit to write an explicit map for the first free-factor embedding $(A,\tau)\hookrightarrow(A,\tau)*(\calZ,\tau_{\calZ})$, and simply identify $(A,\tau)$ with its isomorphic image.
By Lemma~\ref{lem:free-Z-path}, there exists a norm-continuous path
$(q_t)_{t\geq0}\subseteq\mathcal U((\calZ,\tau_{\calZ})*(\calZ,\tau_{\calZ}))$ such that $q_0=1_{(\calZ,\tau_{\calZ})*(\calZ,\tau_{\calZ})}$ and
\[
    \lim_{t\to\infty}
    \|q_t\iota_1(z)q_t^*-\iota_2(z)\|=0
\]
for all $z\in\calZ$.
Let $\alpha\colon\calZ\to A$ be any unital embedding, and note that, since $\calZ$ has a unique tracial state,
$\tau\circ\alpha=\tau_{\calZ}$.
By \cite[Theorem~1.3]{BD01}, we have a unital embedding
$j_\alpha\colon (\calZ,\tau_{\calZ})*(\calZ,\tau_{\calZ})\to (A,\tau)*(\calZ,\tau_{\calZ})$ satisfying
\[
    j_\alpha\circ\iota_1=\alpha,
    \qquad
    j_\alpha\circ\iota_2=\kappa.
\]
For $t\geq0$, we may therefore define
\[
    p_t^\alpha:=j_\psi(q_t)^*j_\alpha(q_t).
\]
This is a norm-continuous unitary path with $p_0^\alpha=1_{(A,\tau)*(\calZ,\tau_{\calZ})}$.
To sum up, we have defined an assignment 
\[
    [\alpha\colon \calZ \to A] \, \mapsto \, (p^{\alpha}_t)_{t\geq0} \subseteq \mathcal U((A,\tau)*(\calZ,\tau_{\calZ})).
\]
For every $z\in\calZ$, we have that
\begin{align*}
    &\|p_t^\alpha\alpha(z)(p_t^\alpha)^*-\psi(z)\|\\
    &\quad=
    \|j_\alpha(q_t)\alpha(z)j_\alpha(q_t)^*
      -j_\psi(q_t)\psi(z)j_\psi(q_t)^*\|\\
    &\quad\leq
    \|j_\alpha(q_t)\alpha(z)j_\alpha(q_t)^*-\kappa(z)\|
    +\|\kappa(z)-j_\psi(q_t)\psi(z)j_\psi(q_t)^*\|\\
    &\quad=
    \|j_\alpha(q_t\iota_1(z)q_t^*-\iota_2(z))\|
    +\|j_\psi(\iota_2(z)-q_t\iota_1(z)q_t^*)\|\\
    &\quad=
    2\|q_t\iota_1(z)q_t^*-\iota_2(z)\|.
\end{align*}
Consequently,
\begin{equation}\label{eq:free-product-path-convergence}
    \|p_t^\alpha\alpha(z)(p_t^\alpha)^*-\psi(z)\|
    \leq 2\|q_t\iota_1(z)q_t^*-\iota_2(z)\|
    \xrightarrow{t\to\infty}0,
\end{equation}
and the estimate is independent of $\alpha$.

Choose increasing nonempty finite subsets $F_n$ of the
unit ball of $\calZ$ whose union is dense,
and let $(\varepsilon_n)_n\subset(0,1)$ be a decreasing sequence converging to zero.
For each $n\geq1$, choose $T_n>0$ such that
\[
    \sup_{t\geq T_n}\max_{z\in F_n}
    \|q_t\iota_1(z)q_t^*-\iota_2(z)\|
    <\frac{\varepsilon_n}{8}.
\]
We claim that there exist a finite subset $G_n\supseteq F_n$ of the unit
ball of $\calZ$ and $0<\delta_n<\varepsilon_n/4$ such that
every unital embedding $\alpha\colon\calZ\to A$ satisfies
\begin{equation}\label{eq:free-product-initial-control}
    \max_{z\in G_n}\|\alpha(z)-\psi(z)\|<\delta_n
    \quad\Longrightarrow\quad
    \sup_{0\leq t\leq T_n}
    \|j_\alpha(q_t)-j_\psi(q_t)\|
    <\frac{\varepsilon_n}{8}.
\end{equation}
Indeed, compactness of $\{q_t\mid 0\leq t\leq T_n\}$ allows
us to choose finitely many elements $b_1,\ldots,b_m$ in
the unital $\ast$-algebra generated by the two free copies of $\mathcal Z$ inside $(\mathcal Z,\tau_{\calZ})\ast(\calZ,\tau_{\calZ})$
such that
\[
    \min_{1\leq k\leq m}\|q_t-b_k\|
    <\frac{\varepsilon_n}{32}
\]
for all $0\leq t\leq T_n$.
Each $b_k$ is a $\ast$-polynomial in finitely many elements $(z^{(1)}_{k,i})_i$
from the first free-factors of $\mathcal Z$ and $(z^{(2)}_{k,i})_i$ from the second free-factor of $\calZ$. Let $G_n$ be the subset of the unit ball of $\calZ$ containing $F_n$ and  $\bigcup_{k,i}\{z^{(1)}_{k,i}\}$ appropriately normalised to elements of the unit ball.
Since the maps $j_\alpha$ and $j_\psi$ agree on the second free-factor of $\mathcal Z$, there exists a $\delta_n>0$ such that, for any unital embedding $\alpha\colon \calZ \to A$,
\[
    \max_{z\in G_n}\|\alpha(z)-\psi(z)\|<\delta_n
    \quad\Longrightarrow\quad
    \max_{1\leq k\leq m}
    \|j_\alpha(b_k)-j_\psi(b_k)\|
    <\frac{\varepsilon_n}{16}.
\]
Then, it follows that 
\begin{align*}
    \sup_{0\leq t\leq T_n}
    \|j_\alpha(q_t)-j_\psi(q_t)\|
    &\leq
    \frac{\varepsilon_n}{16}
    +\max_{1\leq k\leq m}
    \|j_\alpha(b_k)-j_\psi(b_k)\|\\
    &<\frac{\varepsilon_n}{8}.
\end{align*}
After replacing $\delta_n$ with a smaller positive number to make sure that $\delta_n<\varepsilon_n/4$, we have proved the claim in \eqref{eq:free-product-initial-control}.
Fix now $G_n$ and $\delta_n$ as above, and suppose that $\alpha\colon \calZ \to A$ is a unital embedding satisfying
\[
    \max_{z\in G_n}\|\alpha(z)-\psi(z)\|<\delta_n.
\]
Since
\[
    \|p_t^\alpha-1_{(A,\tau)*(\calZ,\tau_{\calZ})}\|
    =
    \|j_\psi(q_t)^*
      (j_\alpha(q_t)-j_\psi(q_t))\|
    =
    \|j_\alpha(q_t)-j_\psi(q_t)\|,
\]
viewing $\alpha$ as a map with range in $(A,\tau)\ast(\calZ,\tau_{\calZ})$, we obtain
\begin{align*}
    &\sup_{0\leq t\leq T_n}\max_{z\in F_n}
    \|p_t^\alpha\alpha(z)(p_t^\alpha)^*-\psi(z)\|\\
    &\quad\leq
    \sup_{0\leq t\leq T_n}\max_{z\in F_n}
    \|p_t^\alpha\alpha(z)(p_t^\alpha)^*-\alpha(z)\|
    +\max_{z\in F_n}\|\alpha(z)-\psi(z)\|\\
    &\quad\leq
    2\sup_{0\leq t\leq T_n}\|p_t^\alpha-1_{(A,\tau)*(\calZ,\tau_{\calZ})}\|
    +\max_{z\in F_n}\|\alpha(z)-\psi(z)\|\\
    &\quad<
    \frac{\varepsilon_n}{4}+\delta_n
    <\frac{\varepsilon_n}{2}.
\end{align*}
For $t\geq T_n$, \eqref{eq:free-product-path-convergence}
and the choice of $T_n$ give
\[
    \sup_{t\geq T_n}\max_{z\in F_n}
    \|p_t^\alpha\alpha(z)(p_t^\alpha)^*-\psi(z)\|
    <\frac{\varepsilon_n}{4}.
\]
Combining these estimates, we have proved that
\begin{equation}\label{eq:free-product-path-control}
    \sup_{t\geq0}\max_{z\in F_n}
    \|p_t^\alpha\alpha(z)(p_t^\alpha)^*-\psi(z)\|
    <\frac{\varepsilon_n}{2}
\end{equation}
whenever $\alpha$ is a unital embedding with
$\max_{z\in G_n}\|\alpha(z)-\psi(z)\|<\delta_n$.

We now construct the path in $A$ that implements strong asymptotic unitary equivalence between $\varphi$ and $\psi$.
By \eqref{eq:free-product-path-convergence}, choose $S_0>0$
such that
\[
    \max_{z\in G_1}
    \|p_{S_0}^{\varphi}\varphi(z)
       (p_{S_0}^{\varphi})^*-\psi(z)\|<\delta_1.
\]
Apply Lemma~\ref{lem:transfer-unitary-path} to
$(p_{S_0s}^{\varphi})_{s\in[0,1]}$, with
$F=\{1_{\calZ}\}$, $H=G_1$, $\varepsilon=1$, and
$\eta=\delta_1$.
We obtain a norm-continuous path
$(u_t)_{t\in[0,1]}\subseteq\mathcal U(A)$ such that $u_0=1_A$
and
\[
    \max_{z\in G_1}\|u_1\varphi(z)u_1^*-\psi(z)\|<\delta_1.
\]

Suppose inductively that the path has been defined on
$[0,n]$ and satisfies
\[
    \max_{z\in G_n}\|u_n\varphi(z)u_n^*-\psi(z)\|<\delta_n.
\]
Put $\alpha_n:=\operatorname{Ad}(u_n)\circ\varphi$.
By \eqref{eq:free-product-path-convergence}, choose
$S_n$ large enough so that
\[
    \max_{z\in G_{n+1}}
    \|p_{S_n}^{\alpha_n}\alpha_n(z)
       (p_{S_n}^{\alpha_n})^*-\psi(z)\|<\delta_{n+1}.
\]
Using \eqref{eq:free-product-path-control} and the estimate above, we may apply
Lemma~\ref{lem:transfer-unitary-path} to
$(p_{S_ns}^{\alpha_n})_{s\in[0,1]}$, with the embeddings
$\alpha_n,\psi$, the sets $F_n,G_{n+1}$, and the tolerances
$\varepsilon_n,\delta_{n+1}$.
We obtain a norm-continuous path
$(w_{n,s})_{s\in[0,1]}\subseteq\mathcal U(A)$ with
$w_{n,0}=1_A$,
\[
    \sup_{s\in[0,1]}\max_{z\in F_n}
    \|w_{n,s}\alpha_n(z)w_{n,s}^*-\psi(z)\|
    <\varepsilon_n,
\]
and
\[
    \max_{z\in G_{n+1}}
    \|w_{n,1}\alpha_n(z)w_{n,1}^*-\psi(z)\|
    <\delta_{n+1}.
\]
We now extend the path $(u_t)_{t\in[0,n]}$ to a path $(u_t)_{t\in[0,n+1]}$ by setting
\[
    u_{n+s}:=w_{n,s}u_n
\]
for $s\in[0,1]$.
Since $w_{n,0}=1_A$, this extension agrees with the
previous path at $n$.
Moreover,
\[
    u_{n+s}\varphi(z)u_{n+s}^*
    =
    w_{n,s}u_n\varphi(z)u_n^*w_{n,s}^*
    =
    w_{n,s}\alpha_n(z)w_{n,s}^*.
\]
Thus, we may conclude that
\[
    \sup_{t\in[n,n+1]}\max_{z\in F_n}
    \|u_t\varphi(z)u_t^*-\psi(z)\|<\varepsilon_n,
\]
and that
\[
    \max_{z\in G_{n+1}}
    \|u_{n+1}\varphi(z)u_{n+1}^*-\psi(z)\|<\delta_{n+1},
\]
which is the next induction step.

The resulting path $(u_t)_{t\geq0}$ is norm continuous
and starts at $1_A$.
Since the sets $F_n$ are increasing with dense union in
the unit ball and $\varepsilon_n\to0$, 
\[
    \lim_{t\to\infty}\|u_t\varphi(z)u_t^*-\psi(z)\|=0
\]
for all $z\in\calZ$.
This proves the proposition.
\end{proof}

We are now ready to prove Theorem \ref{introthm:asymptotic-uniqueness}.
In the purely infinite part of the following proof, we will follow an approach inspired by Lin's work, particularly \cite[Theorem 12.5]{Lin09} and \cite[Theorem 10.5]{Lin11}. In particular, we will need to use a few facts about Kasparov's $KK$-theory \cite{Kasparov} and Skandalis' $KK_{\mathrm{nuc}}$ \cite{Skandalis}, in order to apply Gabe's existence and uniqueness theorems for maps from \cite{Gab24}. 
To know more about these topics, we refer to the aforementioned articles, as well as to \cite{Bla98,JT91} for an introduction to the basic notions of $KK$-theory.
Suffice it to say that, for a separable, unital \Cs-algebra $A$ and a unital \Cs-algebra $B$, a $\ast$-homomorphism $\varphi\colon A \to B$ induces an element $KK(\varphi)$ of the Abelian group $KK(A,B)$, and if $A$ or $B$ is nuclear, then $KK(A,B)\cong KK_{\mathrm{nuc}}(A,B)$.
If $C$ is another separable, unital \Cs-algebra, a theorem of Kasparov \cite{Kasparov} shows that there exists a well-defined Kasparov product $\otimes_C\colon KK(A,C)\times KK(C,B) \to KK(A,B)$, which is, in particular, an associative bilinear map that is compatible with composition of $\ast$-homomorphisms.
If a $KK$-element admits a left- and right-inverse for the Kasparov product, it is referred to as a $KK$-equivalence.
We will also use the following fact.
Under the canonical identification, one has $KK(\mathbb C,B)\cong K_0(B)$, and thus the Kasparov product induces canonical homomorphisms $$\Gamma_i\colon KK(A,B) \to \mathrm{Hom}(K_i(A),K_i(B)) \quad \text{for $i=0,1$}$$
that satisfy $\Gamma_i(KK(\varphi))=K_i(\varphi)$ for any $\ast$-homomorphism $\varphi\colon A\to B$, and $\Gamma_i(x \otimes_C y)=\Gamma_i(y) \circ \Gamma_i(x)$ for all $x\in KK(A,C)$ and $y\in KK(C,B)$.

\begin{proof}[Proof of Theorem \ref{introthm:asymptotic-uniqueness}]

Assume that $(A,\rho)$ is a tracial, selfless \Cs-probability space. Then a straightforward combination of Lemma~\ref{lem:product-Z} with
Proposition~\ref{prop:free-product-asymptotic-uniqueness} proves that any pair of unital emebddings $\calZ \to A$ are strongly asymptotically unitarily equivalent.

Assume now that $\rho$ is not a trace. Then $A$ is purely infinite and simple; this follows from \cite[Theorem 3.1]{Rob25} when $\rho$ is faithful, or \cite[Theorem 1]{Gould26} when $\rho$ is not faithful.

Let $\varphi,\psi\colon\calZ \to A$ be any pair of unital embeddings.
Since the scalar map $\iota\colon \mathbb C \to \calZ$ induces a $KK$-equivalence in $KK(\mathbb C,\calZ)$ by \cite[Lemma 2.11]{JS99}, the following map is an isomorphism 
\[
    \iota^*\colon KK(\calZ,A) \to KK(\mathbb C,A),\quad x\mapsto KK(\iota)\otimes_{\calZ}x.
\]
Hence, it follows from unitality of $\varphi$ and $\psi$ that they induce the same class in $KK(\calZ,A)$. Indeed, 
\[
    \iota^*(KK(\varphi))=KK(\iota) \otimes_{\calZ}KK(\varphi)=KK(\varphi \circ \iota)=KK(\psi \circ \iota)=\iota^*(KK(\psi)),
\]
and applying the inverse of $\iota^*$ yields $KK(\varphi)=KK(\psi)$.
Next, nuclearity of $\calZ$ gives that $KK(\calZ,A)\cong KK_{\mathrm{nuc}}(\calZ,A)$, and that $\varphi$, $\psi$ are nuclear maps.
Thus, combined with the fact that $A$ is simple and purely infinite, we may infer with \cite[Corollary 9.8]{Gab24} that $\varphi$ and $\psi$ are strongly $\cal{O}_\infty$-stable $\ast$-homomorphisms. 
Therefore, Gabe's uniqueness theorem, \cite[Theorem B]{Gab24}, yields a norm-continuous path of unitaries $(r_t)_{t\geq0} \subseteq \mathcal U(A)$ such that 
\[
    \lim_{t\to\infty}\|r_t\varphi(z)r_t^*-\psi(z)\|=0\quad \text{for all }z\in\calZ.
\]
If $r_0=1_A$, then the proof is concluded. 
Hence, from now on we assume that $r_0\neq 1_A$.

Using again \cite[Lemma 2.11]{JS99}, the second-factor embedding $1_{\calZ}\otimes\id_{C(\mathbb T)}\colon C(\mathbb T) \to \calZ\otimes C(\mathbb T)$ induces a $KK$-equivalence in $KK_{\mathrm{nuc}}(C(\mathbb T),\calZ\otimes C(\mathbb T))$.
Let now $\eta\colon C(\mathbb T)\to A$ be the unital $\ast$-homomorphism determined by the assignment $\eta(\id_{\mathbb T})=r_0^*$, and define the $KK$-element 
\[
    x=KK_{\mathrm{nuc}}(1_{\calZ}\otimes\id_{C(\mathbb T)})^{-1} \otimes_{C(\mathbb T)} KK_{\mathrm{nuc}}(\eta) \, \in KK_{\mathrm{nuc}}(\calZ\otimes C(\mathbb T),A).
\]
Clearly, we have that $KK_{\mathrm{nuc}}(1_{\calZ}\otimes\id_{C(\mathbb T)}) \otimes_{\calZ\otimes C(\mathbb T)} x=KK_{\mathrm{nuc}}(\eta)$, and therefore we deduce that 
\begin{align*}
    \Gamma_0(x)\big([1_{\calZ}\otimes 1_{C(\mathbb T)}]_0 \big) &=\Gamma_0(x)\big(K_0(1_{\calZ}\otimes\id_{C(\mathbb T)})([1_{C(\mathbb T)}]_0)\big) \\
    &=\big(\Gamma_0(x)\circ \Gamma_0(KK_{\mathrm{nuc}}(1_{\calZ}\otimes\id_{C(\mathbb T)}))\big)([1_{C(\mathbb T)}]_0) \\
    &=\big(\Gamma_0(KK_{\mathrm{nuc}}(1_{\calZ}\otimes\id_{C(\mathbb T)}) \otimes_{\calZ\otimes C(\mathbb T)} x)\big)([1_{C(\mathbb T)}]_0) \\
    &= K_0(\eta)([1_{C(\mathbb T)}]_0)\\
    &=[1_A]_0.
\end{align*}
It then follows from Gabe's existence theorem, \cite[Theorem A]{Gab24}, that there exists a (strongly $\cal{O}_\infty$-stable) unital $\ast$-homomorphism $\Theta\colon\calZ \otimes C(\mathbb T) \to A$ such that $KK_{\mathrm{nuc}}(\Theta)=x$.
We use $\Theta$ to define the unital $\ast$-homomorphism
\[
    \theta\colon \calZ \to A, \quad \theta=\Theta \circ (\id_{\calZ}\otimes 1_{C(\mathbb T)}),
\]
and set $v=\Theta(1_{\calZ}\otimes\id_{\mathbb T})\in \U(A) \cap \theta(\calZ)'$.
Since  
\[
    KK_{\mathrm{nuc}}(\Theta \circ (1_{\calZ}\otimes\id_{C(\mathbb T)}))=KK_{\mathrm{nuc}}(1_{\calZ}\otimes\id_{C(\mathbb T)}) \otimes_{\calZ\otimes C(\mathbb T)}x=KK_{\mathrm{nuc}}(\eta),
\]
we have that 
\begin{align*}
    [v]_1&=K_1(\Theta \circ (1_{\calZ}\otimes\id_{C(\mathbb T)}))([\id_{\mathbb T}]_1) \\
    &=K_1(\eta)([\id_{\mathbb T}]_1) \\
    &=[r_0^*]_1\\
    &=-[r_0]_1.
\end{align*}
We now apply uniqueness (i.e., \cite[Theorem B]{Gab24}) once more, but this time to the maps $\varphi$ and $\theta$, which we may because $KK_{\mathrm{nuc}}(\varphi)=KK_{\mathrm{nuc}}(\theta)$ from unitality. This yields a norm-continuous path of unitaries $(w_t)_{t\geq0} \subset \U(A)$ such that 
\[
    \lim_{t\to\infty}\|w_t\theta(z)w_t^*-\varphi(z)\|=0\quad \text{for all }z\in\calZ.
\]
Choose $v_t=w_tvw_t^*$, and note that $[v_t,w_t\theta(z)w_t^*]=0$ for all $z\in\calZ$ because $v\in\U(A)\cap\theta(\calZ)'$. Then, we have that 
\begin{align*}
    \|[v_t,\varphi(z)]\|=\|[v_t,\varphi(z)-w_t\theta(z)w_t^*]\| 
    \leq 2\|\varphi(z)-w_t\theta(z)w_t^*\| \xrightarrow{t\to\infty}0
\end{align*}
for all $z\in\calZ$. Also, $[v_t]_1=[v]_1=-[r_0]_1$.
To conclude, choose $u_t=r_tv_t \in\U(A)$, which determines a norm-continuous unitary path.
Observe that 
\begin{align*}
    \|u_t\varphi(z)u_t^*-\psi(z)\| &\leq \|r_t(v_t\varphi(z)v_t^*-\varphi(z))r_t^*\| + \|r_t\varphi(z)r_t^*-\psi(z)\| \\
    &=\|[v_t,\varphi(z)]\|+\|r_t\varphi(z)r_t^*-\psi(z)\| \xrightarrow{t\to\infty}0
\end{align*}
for all $z\in\calZ$. 
Moreover, $[u_0]_1=[r_0v_0]_1=[r_0]_1+[v_0]_1=0$. Since $A$ is purely infinite and simple, \cite[Theorem 1.9]{Cun81} gives that $\U(A)/\U_0(A)\cong K_1(A)$, and we conclude that $u_0$ is connected to $1_A$ by a norm-continuous path of unitaries. Attaching $(u_t)_{t\geq0}$ to the path that connects $u_0$ to the unit, we obtain a unitary path witnessing strong asymptotic unitary equivalence between $\varphi$ and $\psi$, as desired. 
\end{proof}

\bibliography{references.bib}

@article{PP13,
  author  = {Petrovi{\'c}, Zoran Z. and
             Prvulovi{\'c}, Branislav I.},
  title   = {Equivariant maps between complex {Stiefel} manifolds},
  journal = {Osaka J. Math.},
  volume  = {50},
  number  = {1},
  year    = {2013},
  pages   = {187--196}
}

@article{Rob12,
  author  = {Robert, Leonel},
  title   = {Classification of inductive limits of
             1-dimensional {NCCW} complexes},
  journal = {Adv. Math.},
  volume  = {231},
  number  = {5},
  pages   = {2802--2836},
  year    = {2012}
}

@article{APS11,
  author  = {Antoine, Ramon and Perera, Francesc
             and Santiago, Luis},
  title   = {Pullbacks, {$C(X)$}-algebras,
             and their {Cuntz} semigroup},
  journal = {J. Funct. Anal.},
  volume  = {260},
  number  = {10},
  pages   = {2844--2880},
  year    = {2011}
}

@article{ERS11,
  author  = {Elliott, George A. and Robert, Leonel and Santiago, Luis},
  title   = {The cone of lower semicontinuous traces on a {\Cs}-algebra},
  journal = {Amer. J. Math.},
  volume  = {133},
  number  = {4},
  year    = {2011},
  pages   = {969--1005}
}

@article{ABPP14,
  author  = {Antoine, Ramon and Bosa, Joan
             and Perera, Francesc and Petzka, Henning},
  title   = {Geometric structure of dimension functions
             of certain continuous fields},
  journal = {J. Funct. Anal.},
  volume  = {266},
  number  = {4},
  pages   = {2403--2423},
  year    = {2014}
}

@article{DW09,
  author  = {Dadarlat, Marius and Winter, Wilhelm},
  title   = {On the {$KK$}-theory of strongly
             self-absorbing {\Cs}-algebras},
  journal = {Math. Scand.},
  volume  = {104},
  number  = {1},
  pages   = {95--107},
  year    = {2009}
}

@article{FG,
  author  = {Fima, Pierre and Germain, Emmanuel},
  title   = {The {$KK$}-theory of amalgamated free products},
  journal = {Adv. Math.},
  volume  = {369},
  note   = {article 107174},
  year    = {2020}
}

@article{Rob25,
  author  = {Robert, Leonel},
  title   = {Selfless {\Cs}-algebras},
  journal = {Adv. Math.},
  volume  = {478},
  note   = {article 110409},
  year    = {2025}
}

@article{APT18,
  author  = {Antoine, Ramon and Perera, Francesc and
             Thiel, Hannes},
  title   = {Tensor products and regularity properties
             of {Cuntz} semigroups},
  journal = {Mem. Amer. Math. Soc.},
  volume  = {251},
  number  = {1199},
  year    = {2018}
}

@article{GP26,
  author  = {Gardella, Eusebio and Perera, Francesc},
  title   = {The modern theory of {Cuntz} semigroups
             of {\Cs}-algebras},
  journal = {EMS Surv. Math. Sci.},
  volume  = {13},
  number  = {1},
  year    = {2026},
  pages   = {133--214}
}

@article{Rie83,
  title={Dimension and stable rank in the $K$-theory of \Cs-algebras},
  author={Rieffel, Marc A},
  journal={Proc. London Math. Soc.},
  volume={3},
  number={2},
  pages={301--333},
  year={1983},
  publisher={Wiley Online Library}
}

@article{Win11,
  author    = {Wilhelm Winter},
  title     = {Strongly self-absorbing \Cs-algebras are $\mathcal{Z}$-stable},
  journal   = {J. Noncommut. Geom.},
  volume    = {5},
  number    = {2},
  pages     = {253--264},
  year      = {2011},
}

@article{Cun77,
  title={Simple \Cs-algebraa generated by isometries},
  author={Cuntz, Joachim},
  journal={Comm. Math. Phys.},
  volume={57},
  number={2},
  pages={173--185},
  year={1977},
  publisher={Springer}
}

@article{Haa14,
  author        = {Haagerup, Uffe},
  title         = {Quasitraces on exact \Cs-algebras are traces},
  journal       = {C. R. Math. Rep. Acad. Sci. Canada},
  volume        = {36},
  number        = {2-3},
  pages         = {67--92},
  year          = {2014},
}

@article{Gab24,
  author  = {Gabe, James},
  title   = {Classification of {$\mathcal O_\infty$}-stable {\Cs}-algebras},
  journal = {Mem. Amer. Math. Soc.},
  volume  = {293},
  number  = {1461},
  year    = {2024},
  pages   = {v+115}
}

@article{Cun81,
  author  = {Cuntz, Joachim},
  title   = {{$K$}-theory for certain {\Cs}-algebras},
  journal = {Ann. of Math. (2)},
  volume  = {113},
  number  = {1},
  year    = {1981},
  pages   = {181--197}
}

@incollection{Kir95,
  author    = {Kirchberg, Eberhard},
  title     = {Exact {\Cs}-algebras, tensor products, and the
               classification of purely infinite algebras},
  editor    = {Chatterji, S. D.},
  booktitle = {Proc. ICM, Vol.~1, 2 ({Z\"urich}, 1994)},
  publisher = {Birkh{\"a}user},
  address   = {Basel},
  year      = {1995},
  pages     = {943--954}
}

@article{Gab20,
  author  = {Gabe, James},
  title   = {A new proof of {Kirchberg}'s
             {$\mathcal O_2$}-stable classification},
  journal = {J. Reine Angew. Math.},
  volume  = {761},
  year    = {2020},
  pages   = {247--289}
}

@article{TW09,
  author  = {Toms, Andrew S. and Winter, Wilhelm},
  title   = {The {Elliott} conjecture for {Villadsen} algebras of the first type},
  journal = {J. Funct. Anal.},
  volume  = {256},
  number  = {5},
  year    = {2009},
  pages   = {1311--1340},
}

@article{Vil98,
  author  = {Villadsen, Jesper},
  title   = {Simple {\Cs}-algebras with perforation},
  journal = {J. Funct. Anal.},
  volume  = {154},
  number  = {1},
  year    = {1998},
  pages   = {110--116},
}

@misc{BG24,
  author = {Bouwen, Ben and Gabe, James},
  title  = {A unified approach for classifying simple nuclear
            {\Cs}-algebras},
  year   = {2024},
  note   = {preprint},
  url    = {https://arxiv.org/abs/2412.15968}
}

@article{KP00,
  author    = {Kirchberg, Eberhard and Phillips, N. Christopher},
  title     = {Embedding of exact {\Cs}-algebras in the
               {Cuntz} algebra {$\mathcal O_2$}},
  journal   = {J. Reine Angew. Math.},
  volume    = {525},
  year      = {2000},
  pages     = {17--53}
}

@article{Phi00,
  author    = {Phillips, N. Christopher},
  title     = {A classification theorem for nuclear purely infinite
               simple {\Cs}-algebras},
  journal   = {Doc. Math.},
  volume    = {5},
  year      = {2000},
  pages     = {49--114}
}

@incollection{Boc97,
  author    = {Boca, Florin P.},
  title     = {A note on full free product {\Cs}-algebras,
               lifting and quasidiagonality},
  booktitle = {Operator Theory, Operator Algebras and Related Topics
               ({Timi\c{s}oara}, 1996)},
  publisher = {Theta Foundation},
  address   = {Bucharest},
  year      = {1997},
  pages     = {51--63},
  url       = {https://arxiv.org/abs/1508.03123}
}

@incollection{Bla88,
    booktitle={Operator Algebras and Applications}, 
  title={Comparison theory for simple \Cs-algebras},
  author={Blackadar, Bruce},
  journal={London Math. Soc. Lecture Note Ser. Cambridge Univ. Press, Cambridge},
  number={1},
volume={135},
  pages={21--54},
  year={1988}
}

@article{DR98,
  author  = {Dykema, Kenneth J. and R{\o}rdam, Mikael},
  title   = {Projections in free product {\Cs}-algebras},
  journal = {Geom. Funct. Anal.},
  volume  = {8},
  number  = {1},
  year    = {1998},
  pages   = {1--16}
}

@article{DR00,
  author  = {Dykema, Kenneth J. and R{\o}rdam, Mikael},
  title   = {Projections in free product {\Cs}-algebras, {II}},
  journal = {Math. Z.},
  volume  = {234},
  number  = {1},
  year    = {2000},
  pages   = {103--113}
}

@article{Pop95,
  author  = {Popa, Sorin},
  title   = {Free-independent sequences in type {$\mathrm{II}_1$}
             factors and related problems},
  journal = {Ast{\'e}risque},
  number  = {232},
  year    = {1995},
  pages   = {187--202}
}

@article{Ger97,
  author  = {Germain, Emmanuel},
  title   = {{$KK$}-theory of the full free product
             of unital {\Cs}-algebras},
  journal = {J. reine angew. Math.},
  volume  = {485},
  pages   = {1--10},
  year    = {1997}
}

@article{Tho03,
  author  = {Thomsen, Klaus},
  title   = {On the {$KK$}-theory and the {$E$}-theory
             of amalgamated free products of {\Cs}-algebras},
  journal = {J. Funct. Anal.},
  volume  = {201},
  number  = {1},
  pages  = {30--56},
  year={2003},
}

@article{Cun78,
  author  = {Cuntz, Joachim},
  title   = {Dimension functions on simple {\Cs}-algebras},
  journal = {Math. Ann.},
  volume  = {233},
  number  = {2},
  year    = {1978},
  pages   = {145--153},
}

@article{CEI08,
  author  = {Coward, Kristofer T. and Elliott, George A.
             and Ivanescu, Cristian},
  title   = {The {Cuntz} semigroup as an invariant
             for {\Cs}-algebras},
  journal = {J. reine angew. Math.},
  volume  = {623},
  year    = {2008},
  pages   = {161--193},
}

@article{JS99,
  author  = {Jiang, Xinhui and Su, Hongbing},
  title   = {On a simple unital projectionless {\Cs}-algebra},
  journal = {Amer. J. Math.},
  volume  = {121},
  number  = {2},
  year    = {1999},
  pages   = {359--413},
}

@misc{HKEPR25,
  author        = {Hayes, Ben and Kunnawalkam Elayavalli, Srivatsav
                   and Patchell, Gregory and Robert, Leonel},
  title         = {Selfless inclusions of {\Cs}-algebras},
  year          = {2025},
  journal = {J. Funct. Anal.},
  note        = {to appear in Acta Math.},
  url = "https://arxiv.org/pdf/2510.13398",
}

@article{BD01,
  author  = {Blanchard, Etienne F. and Dykema, Kenneth J.},
  title   = {Embeddings of reduced free products of operator algebras},
  journal = {Pacific J. Math.},
  volume  = {199},
  number  = {1},
  year    = {2001},
  pages   = {1--19}
}

@article{BS16,
  author  = {Barlak, Sel{\c{c}}uk and Szab{\'o}, G{\'a}bor},
  title   = {Sequentially split {$\ast$}-homomorphisms between
             {\Cs}-algebras},
  journal = {Internat. J. Math.},
  volume  = {27},
  number  = {13},
  year    = {2016},
  note   = {article 1650105}
}

@article{RW10,
  author  = {R{\o}rdam, Mikael and Winter, Wilhelm},
  title   = {The {Jiang--Su} algebra revisited},
  journal = {J. reine angew. Math.},
  volume  = {642},
  year    = {2010},
  pages   = {129--155},
}

@article{Ell07,
  author  = {Elliott, George A.},
  title   = {{Hilbert} modules over a {\Cs}-algebra of stable rank one},
  journal = {C. R. Math. Acad. Sci. Soc. R. Can.},
  volume  = {29},
  number  = {2},
  year    = {2007},
  pages   = {48--51},
}

@article{FHLRTVW21,
  author  = {Farah, Ilijas and Hart, Bradd and Lupini, Martino
             and Robert, Leonel and Tikuisis, Aaron
             and Vignati, Alessandro and Winter, Wilhelm},
  title   = {Model theory of {\Cs}-algebras},
  journal = {Mem. Amer. Math. Soc.},
  volume  = {271},
  number  = {1324},
  year    = {2021},
  pages   = {viii+127}
}

@article{Rob26,
  title={Strict comparison and selflessness},
  author={Robert, Leonel},
  journal={arXiv preprint arXiv:2609.15776},
  year={2026}
}

@article{Ror04,
  author  = {R{\o}rdam, Mikael},
  title   = {The stable and the real rank of {$\mathcal Z$}-absorbing {\Cs}-algebras},
  journal = {Internat. J. Math.},
  volume  = {15},
  number  = {10},
  year    = {2004},
  pages   = {1065--1084}
}

@article{CETWW21,
  author  = {Castillejos, Jorge and Evington, Samuel and Tikuisis, Aaron and White, Stuart and Winter, Wilhelm},
  title   = {Nuclear dimension of simple {\Cs}-algebras},
  journal = {Invent. Math.},
  volume  = {224},
  number  = {1},
  year    = {2021},
  pages   = {245--290}
}

@article{Win14,
  author  = {Winter, Wilhelm},
  title   = {Localizing the {Elliott} conjecture at strongly self-absorbing {\Cs}-algebras},
  journal = {J. reine angew. Math.},
  volume  = {692},
  year    = {2014},
  pages   = {193--231},
  note    = {With an appendix by Huaxin Lin}
}

@article{Win16,
  author  = {Winter, Wilhelm},
  title   = {Classifying crossed product {\Cs}-algebras},
  journal = {Amer. J. Math.},
  volume  = {138},
  number  = {3},
  year    = {2016},
  pages   = {793--820}
}

@article{TWW17,
  author  = {Tikuisis, Aaron and White, Stuart and Winter, Wilhelm},
  title   = {Quasidiagonality of nuclear {\Cs}-algebras},
  journal = {Ann. of Math. (2)},
  volume  = {185},
  number  = {1},
  year    = {2017},
  pages   = {229--284}
}

@article{GLN20a,
  author  = {Gong, Guihua and Lin, Huaxin and Niu, Zhuang},
  title   = {A classification of finite simple amenable {$\mathcal Z$}-stable {\Cs}-algebras, {I}: {\Cs}-algebras with generalized tracial rank one},
  journal = {C. R. Math. Acad. Sci. Soc. R. Can.},
  volume  = {42},
  number  = {3},
  year    = {2020},
  pages   = {63--450}
}

@article{GLN20b,
  author  = {Gong, Guihua and Lin, Huaxin and Niu, Zhuang},
  title   = {A classification of finite simple amenable {$\mathcal Z$}-stable {\Cs}-algebras, {II}: {\Cs}-algebras with rational generalized tracial rank one},
  journal = {C. R. Math. Acad. Sci. Soc. R. Can.},
  volume  = {42},
  number  = {4},
  year    = {2020},
  pages   = {451--539}
}

@article{EGLN25,
  author  = {Elliott, George A. and Gong, Guihua and Lin, Huaxin and Niu, Zhuang},
  title   = {On the classification of simple amenable {\Cs}-algebras with finite decomposition rank, {II}},
  journal = {J. Noncommut. Geom.},
  volume  = {19},
  number  = {1},
  year    = {2025},
  pages   = {73--104}
}

@misc{CGSTW23,
  author  = {Carri{\'o}n, Jos{\'e} R. and Gabe, James and Schafhauser, Christopher and Tikuisis, Aaron and White, Stuart},
  title   = {Classifying {$*$}-homomorphisms {I}: Unital simple nuclear {\Cs}-algebras},
  year    = {2023},
  journal    = {Acta Math.},
  note = {to appear in Acta Math.},
  url     = {https://arxiv.org/abs/2307.06480}
}

@article{Avi82,
  author  = {Avitzour, Daniel},
  title   = {Free products of {\Cs}-algebras},
  journal = {Trans. Amer. Math. Soc.},
  volume  = {271},
  number  = {2},
  year    = {1982},
  pages   = {423--435}
}

@book{Gol23,
  editor    = {Goldbring, Isaac},
  title     = {Model Theory of Operator Algebras},
  series    = {De Gruyter Ser. Log. Appl.},
  volume    = {11},
  publisher = {De Gruyter},
  address   = {Berlin--Boston},
  year      = {2023}
}

@incollection{Voi85,
  author    = {Voiculescu, Dan},
  title     = {Symmetries of some reduced free product {\Cs}-algebras},
  booktitle = {Operator Algebras and their Connections with
               Topology and Ergodic Theory
               (Bu{\c{s}}teni, 1983)},
  series    = {Lecture Notes in Math.},
  volume    = {1132},
  publisher = {Springer},
  address   = {Berlin},
  year      = {1985},
  pages     = {556--588}
}

@article{AGKP25,
  author  = {Amrutam, Tattwamasi and Gao, David and
             Kunnawalkam Elayavalli, Srivatsav and Patchell, Gregory},
  title   = {Strict comparison in reduced group {\Cs}-algebras},
  journal = {Invent. Math.},
  volume  = {242},
  number  = {3},
  year    = {2025},
  pages   = {639--657}
}

@misc{HKR25,
  author = {Hayes, Ben and Kunnawalkam Elayavalli, Srivatsav and
            Robert, Leonel},
  title  = {Selfless reduced free product {\Cs}-algebras},
  year   = {2025},
  note   = {preprint},
  url    = {https://arxiv.org/abs/2505.13265}
}

@misc{Ozawa25,
  author = {Ozawa, Narutaka},
  title  = {Proximality and selflessness for group {\Cs}-algebras},
  year   = {2025},
  note   = {preprint},
  url    = {https://arxiv.org/abs/2508.07938}
}

@misc{Vig25,
  author = {Vigdorovich, Itamar},
  title  = {Structural properties of reduced {\Cs}-algebras
            associated with higher-rank lattices},
  year   = {2025},
  note   = {preprint},
  url    = {https://arxiv.org/abs/2503.12737}
}

@article{Vig26,
  author  = {Vigdorovich, Itamar},
  title   = {Selfless reduced {\Cs}-algebras of linear groups},
  journal = {Proc. Lond. Math. Soc.},
  volume  = {133},
  number  = {1},
  year    = {2026},
  pages   = {e70180}
}

@misc{RTV25,
  author = {Raum, Sven and Thiel, Hannes and Vilalta, Eduard},
  title  = {Strict comparison for twisted group {\Cs}-algebras},
  year   = {2025},
  note   = {preprint},
  url    = {https://arxiv.org/abs/2505.18569}
}

@article{FFKP26a,
  author  = {Flores, Felipe and Klisse, Mario and
             {\'{O} Cobhthaigh}, M{\'i}che{\'a}l and Pagliero, Matteo},
  title   = {Selfless reduced free products and graph products
             of {\Cs}-algebras},
  journal = {Adv. Math.},
  volume  = {502},
  number  = {C},
  year    = {2026},
  note   = {article 111190}
}

@misc{FFKP26b,
  author = {Flores, Felipe and Klisse, Mario and
            {\'{O} Cobhthaigh}, M{\'i}che{\'a}l and Pagliero, Matteo},
  title  = {Pureness and stable rank one for reduced twisted group
            {\Cs}-algebras of certain group extensions},
  year   = {2026},
  note   = {preprint},
  url    = {https://arxiv.org/abs/2601.19758}
}

@misc{KPT25,
  author = {Kunnawalkam Elayavalli, Srivatsav and Patchell, Gregory
            and Teryoshin, Lizzy},
  title  = {Some remarks on decay in countable groups and
            amalgamated free products},
  year   = {2025},
  note   = {preprint},
  url    = {https://arxiv.org/abs/2509.08754}
}

@misc{GKPT26,
  author = {Gao, David and Kunnawalkam Elayavalli, Srivatsav and
            Patchell, Gregory and Teryoshin, Lizzy},
  title  = {Selfless reduced amalgamated free products and
            {HNN} extensions},
  year   = {2026},
  note   = {preprint},
  url    = {https://arxiv.org/abs/2604.06982}
}

@misc{BF26,
  author = {Basu, Aaratrick and Flores, Felipe},
  title  = {Selfless inclusions arising from commensurator groups
            of hyperbolic groups},
  year   = {2026},
  note   = {preprint},
  url    = {https://arxiv.org/abs/2605.12737}
}

@misc{Omland26,
  author = {Omland, Tron},
  title  = {Selflessness for twisted group {\Cs}-algebras of
            amenable groups and their inclusions},
  year   = {2026},
  note   = {preprint},
  url    = {https://arxiv.org/abs/2606.27198}
}

@misc{Gould26,
  author = {Gould, Miles},
  title  = {The selfless dichotomy},
  year   = {2026},
  note   = {preprint},
  url    = {https://arxiv.org/abs/2606.09654}
}

@misc{GJKPR26,
  author = {Gao, David and Junge, Marius and
            Kunnawalkam Elayavalli, Srivatsav and Patchell, Gregory
            and Robert, Leonel},
  title  = {Selfless {\Cs}-correspondences, operator valued
            {\Cs}-probability spaces and completely positive maps},
  year   = {2026},
  note   = {preprint},
  url    = {https://arxiv.org/abs/2607.20361}
}

@misc{Arimoto26,
  author = {Arimoto, Ryoya},
  title  = {{\Cs}-selflessness of vigorous groups},
  year   = {2026},
  note   = {preprint},
  url    = {https://arxiv.org/abs/2609.29399}
}

@misc{FG26boundary,
  author = {Flores, Felipe and Gondek, Joseph},
  title  = {Simplicity and pure infiniteness for {\Cs}-algebras
            associated with stabilizers of boundary actions},
  year   = {2026},
  note   = {preprint},
  url    = {https://arxiv.org/abs/2608.30127}
}

@misc{STW25,
  author = {Schafhauser, Christopher and Tikuisis, Aaron and
            White, Stuart},
  title  = {Nuclear {\Cs}-algebras: 99 problems},
  year   = {2025},
  journal = {Münster J. Math.},
  note   = {to appear in Münster J. Math.},
  url    = {https://arxiv.org/abs/2506.10902}
}

@misc{Yang25,
  author = {Yang, Wenyuan},
  title  = {An extreme boundary of acylindrically hyperbolic groups},
  year   = {2025},
  note   = {preprint},
  url    = {https://arxiv.org/abs/2511.16400}
}

@article{Lin09,
  author = {Lin, Huaxin},
  title = {Asymptotically unitary equivalence and asymptotically inner automorphisms},
  journal = {Amer. J. Math.},
  volume = {131},
  number = {6},
  year = {2009},
  pages = {1589--1677},
}

@article{Kasparov,
  author = {Kasparov, Gennadi G.},
  title = {The operator {$K$}-functor and extensions of {\Cs}-algebras},
  journal = {Izv. Akad. Nauk SSSR Ser. Mat.},
  volume = {44},
  number = {3},
  year = {1980},
  pages = {571--636},
}

@book{Bla98,
  author = {Blackadar, Bruce},
  title = {{$K$}-theory for operator algebras},
  edition = {2},
  series = {Mathematical Sciences Research Institute Publications},
  volume = {5},
  publisher = {Cambridge University Press},
  address = {Cambridge},
  year = {1998},
}

@book{JT91,
  author = {Jensen, Kjeld Knudsen and Thomsen, Klaus},
  title = {Elements of {$KK$}-theory},
  series = {Mathematics: Theory \& Applications},
  publisher = {Birkh{\"a}user},
  address = {Boston, MA},
  year = {1991},
}

@article{Skandalis,
  author = {Skandalis, Georges},
  title = {Une notion de nucl{\'e}arit{\'e} en {$K$}-th{\'e}orie (d'apr{\`e}s {J. Cuntz})},
  journal = {$K$-Theory},
  volume = {1},
  number = {6},
  year = {1988},
  pages = {549--573},
}

@article{Lin11,
  author = {Lin, Huaxin},
  title = {Asymptotic unitary equivalence and classification of simple amenable {\Cs}-algebras},
  journal = {Invent. Math.},
  volume = {183},
  number = {2},
  year = {2011},
  pages = {385--450},
}
\bibliographystyle{mybstnum.bst}

\end{document}